\documentclass[12pt]{article}
\usepackage{graphicx}
\usepackage{amsmath}
\usepackage{amsfonts}
\usepackage{amssymb}
\usepackage{color}
\providecommand{\U}[1]{\protect\rule{.1in}{.1in}}
\newtheorem{theorem}{Theorem}

\newtheorem{lemma}[theorem]{Lemma}

\newtheorem{proposition}[theorem]{Proposition}

\newenvironment{proof}[1][Proof]{\noindent\textbf{#1.} }{\ \rule{0.5em}{0.5em}}
\def\myblacksquare{\rule{1.2ex}{1.2ex}}

\begin{document}

$\ $

\vspace{2.cm}

\begin{center}
{\Large \textbf{Extremal Cylinder Configurations I: \\ \vskip .2cm Configuration $C_{\mathfrak{m}}$}}

\vspace{.4cm} {\large \textbf{Oleg Ogievetsky$^{\diamond\,\ast}$\footnote{Also
at Lebedev Institute, Moscow, Russia.} and Senya Shlosman$^{^{\diamond}
\,\dag\,\ddagger}$}}

\vskip .3cm $^{\diamond}$Aix Marseille Universit\'{e}, Universit\'{e} de
Toulon, CNRS, \\ CPT UMR 7332, 13288, Marseille, France

\vskip .05cm $^{\dag}$Inst. of the Information Transmission Problems, RAS,
Moscow, Russia

\vskip .05cm $^{\ddagger}$ Skolkovo Institute of Science and Technology,
Moscow, Russia

\vskip .05cm $^{\ast}${Kazan Federal University, Kremlevskaya 17, Kazan
420008, Russia}
\end{center}

\vskip .6cm
\hfill {\sf I do not ask for better than not to be believed.}

\vskip .1cm
\hfill Axel Munthe, {\it The story of San Michele}

$\ $
\vskip .5cm

\begin{abstract}\noindent
We study the path $\Gamma=\{ C_{6,x}\ \vert\ x\in [0,1]\}$ in the moduli space of configurations of 6 equal cylinders touching the unit
sphere. Among the configurations $C_{6,x}$ is the record configuration
$C_{\mathfrak{m}}$ of \cite{OS}. We show that $C_{\mathfrak{m}}$ is a local sharp maximum of
the distance function, so in particular the configuration $C_{\mathfrak{m}}$ is not only unlockable but rigid.
We show that if $\frac{(1 + x) (1 + 3 x)}{3}$ is a rational number but not a square of a rational number, the configuration $C_{6,x}$
has some hidden symmetries, part of which we explain.
\end{abstract}

\newpage
\tableofcontents

\newpage
\section{Introduction}
This is a continuation of our work \cite{OS}. In that paper we were considering the configurations
of six (infinite) nonintersecting cylinders of the same radius $r$ touching the unit sphere
$\mathbb{S}^{2}\subset\mathbb{R}^{3}$.
We were interested in the maximal value
of $r$ for which this is possible. We have constructed in \cite{OS} the `record'
configuration $C_{\mathfrak{m}}$
of six cylinders of radius
\begin{equation}
r_{\mathfrak{m}}=\frac{1}{8}\left(  3+\sqrt{33}\right)  \approx
1.093070331,\label{30}
\end{equation}
thus we know that the maximal value of $r$ is at least $r_{\mathfrak{m}}.$ We
believe that $r_{\mathfrak{m}}$ is in fact the maximal possible value for $r$,
but we have no proof of that.

\vskip .2cm
In \cite{OS} we have constructed the deformation $C_{6,x}$ of the configuration $C_6$ of
six vertical unit nonintersecting cylinders.
The configuration $C_6$ corresponds to $x=1$ while $C_{\mathfrak{m}}$ -- to $x=1/2$. These configurations are shown on Figure \ref{confCcyl}
(the green unit ball is in the center).
\begin{figure}[th]
\centering
$\ \ \ \ $\raisebox{.3cm}{\includegraphics[scale=0.71]{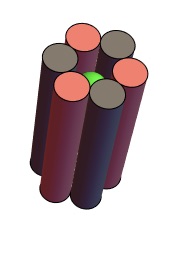}}$\ \ \ \ $\includegraphics[scale=0.352]{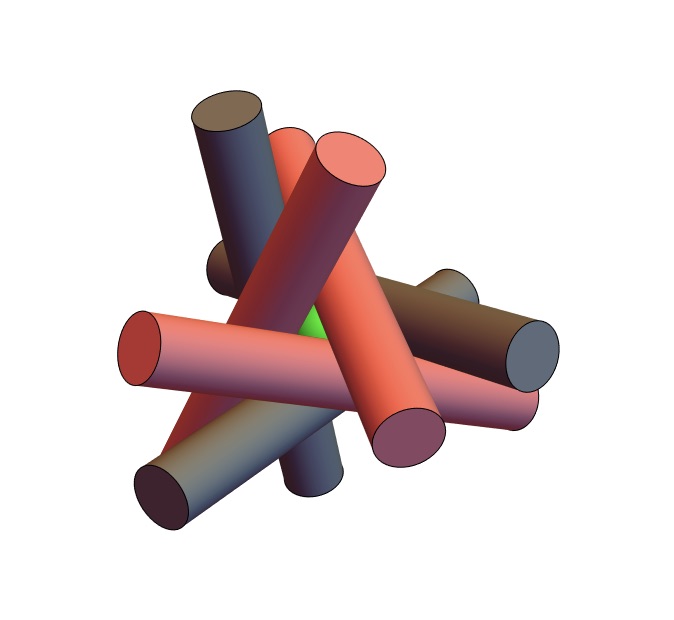}
\parbox{11.8cm}{
\caption{Two configurations of cylinders: the configuration $C_6$ of six parallel cylinders of radius 1 (on the left) and the configuration $C_{\mathfrak{m}}$ of six cylinders of radius $\,\approx\! 1.0931$ (on the right) } }
\label{confCcyl}
\end{figure}

\vskip.2cm
To explain the results of the present paper, we introduce some notation. A cylinder $\varsigma$
touching the unit sphere $\mathbb{S}^{2}$ has a unique generator (a line
parallel to the axis of the cylinder) $\iota(\varsigma)$ touching
$\mathbb{S}^{2}$. We will usually represent a configuration $\{\varsigma
_{1},\dots,\varsigma_{L}\}$ of cylinders touching the unit sphere by the
configuration $\{\iota(\varsigma_{1}),\dots,\iota(\varsigma_{L})\}$ of tangent
to $\mathbb{S}^{2}$ lines. The manifold of all such six-tuples we denote
by $M^{6}.$

\vskip .2cm
For example, let $C_{6}\equiv C_{6}\left(  0,0,0\right)  $ be the
configuration of six nonintersecting cylinders of radius $1,$ parallel to the
$z$ direction in $\mathbb{R}^{3}$ and touching the unit ball centered at the
origin. The configuration of tangent lines associated to the configuration $C_{6}$ is shown on Figure \ref{confC6tan}.

\begin{figure}[th]
\vspace{.4cm} \centering
\includegraphics[scale=0.28]{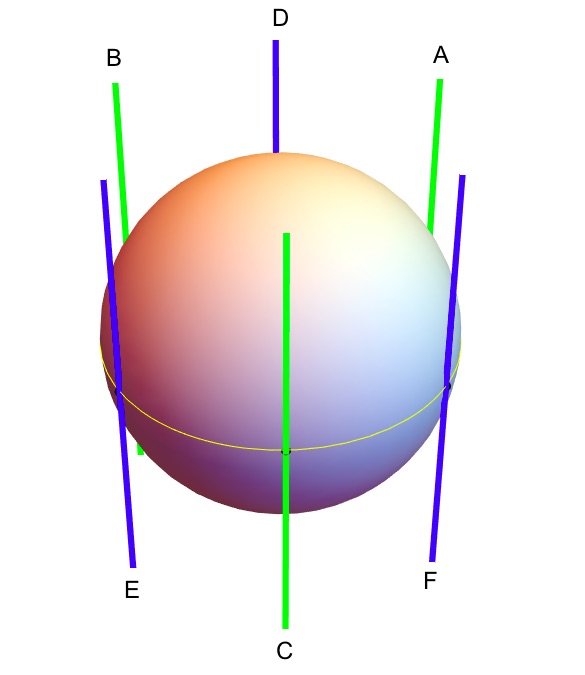}
\caption{Configuration $C_{6}$ of tangent lines}
\label{confC6tan}
\end{figure}

\vskip .2cm
Let $\varsigma^{\prime},\varsigma^{\prime\prime}$ be two equal cylinders of
radius $r$ touching $\mathbb{S}^{2},$ which also touch each other, while
$\iota^{\prime},\iota^{\prime\prime}$ are the corresponding tangents to
$\mathbb{S}^{2}.$ If $d=d_{\iota^{\prime}\iota^{\prime\prime}}$ is the
distance between $\iota^{\prime},\iota^{\prime\prime}$ then we have
\[
r=\frac{d}{2-d},
\]
so it is really the same - to study the manifold of six-tuples of cylinders of
equal radii, some of which are touching, or to study the manifold $M^{6}$ and
the function $D$ on it:

\[
D\left(  \iota_{1},\ldots ,\iota_{6}\right)  =\min_{1\leq i<j\leq6}d_{\iota
_{i}\iota_{j}}.
\]

\vskip .2cm
In this paper we prove that the configuration $C_{\mathfrak{m}}$ is a sharp local maximum of the function $D$.
In the process of the proof we also
show that ($\operatorname{mod}SO\left(  3\right)  $) the 15-dim tangent space
at $C_{\mathfrak{m}}$ contains a 4-dimensional subspace
along which the function $D\left(  \mathbf{m}\right)^2  $ decays quadratically,
while along any other tangent direction it decays linearly.

\vskip .2cm
It turns out that the question of finding sufficient conditions for the extrema of the $min$ functions can be
quite delicate.

\vskip .2cm
For the configuration $C_{\mathfrak{m}}$ we distinguish twelve relevant distances $l_1$, $\dots$, $l_{12}$ out of the total of fifteen
pairwise distances between the tangent lines. Thus the question is about the local maximum of the non-analytic
function (the minimum of the squares $F_1$, $\dots$, $F_{12}$ of these twelve distances) in fifteen variables.

\vskip .2cm
We make a general remark. Let $F_{1},\ldots ,F_{m}$ be analytic functions in $n$ variables, $n\geq m$, and let
\[ {\sf F}\left(  x\right)  :=\min\left\{  F_{1}\left(  x\right)
,\dots,F_{m}\left(x\right)  \right\} \ .\]
We assume that $F_j(0)=0$, $j=1,\dots,m$.
Suppose that the point $0\in\mathbb{R}^{n}$ is a local maximum of the function ${\sf F}\left(  x\right)$.
Then the differentials $dF_{1},\ldots ,dF_{m}$ are necessarily linearly dependent at $0$. Indeed, let $\Pi_j^+$, $j=1,\dots,m$,
be the half-space in $\mathbb{R}^{n}$ on which the differential $dF_j(0)$ is positive. If the differentials $dF_{1},\ldots ,dF_{m}$ are
independent at $0$ then the intersection of the half-spaces $\Pi_j$ is non-empty, so there is a direction from $0$
along which all $m$ functions $F_{i}$ are increasing, thus the point $0$ is not a local maximum of the function ${\sf F}\left(  x\right)$.
This remark is a generalization of the case $m=1$ (just one analytic function): if the point $0$ is a local maximum of an analytic function $F(x)$ then
its differential vanishes at the point $0$, $dF(0)=0$.

\vskip .2cm
We return for a moment to the configuration $C_{\mathfrak{m}}$. We calculate explicitly the differentials of the squares of the twelve relevant
distances $l_1$, $\dots$, $l_{12}$. Our first observation is that they are indeed not linearly independent.  More precisely, there is a \textit{single} linear
combination $\lambda$ of the differentials which vanish, $\lambda(dl_1^2,\dots,dl_{12}^2)=0$.

\vskip .2cm
We continue the general remark. Suppose that the point $0\in\mathbb{R}^{n}$ is a local maximum of the function ${\sf F}\left(  x\right)$
and there is exactly one linear dependency between the differentials $dF_{1},\ldots ,dF_{m}$ at $0$. Then this dependency {\it must} be convex,
in the sense that it must have the form $\lambda^1 dF_{1}+\ldots \lambda^m dF_{m}=0$ with $\lambda^j>0$, $j=1,\dots,m$.  Indeed, let us
assume that the single linear combination of the differentials is not convex. Then, renumbering, if necessary, the functions $F_j(x)$, we write the linear
dependency between the differentials in the following form
$$dF_1(0)=\mu^2 dF_2(0)\pm\dots \pm\mu^m dF_m(0)\ ,$$
with $\mu^j>0$, $j=2,\dots,m$. The differentials $dF_{2},\ldots ,dF_{m}$ are independent and the subset of $\mathbb{R}^{n}$
where $\mu^2 dF_2(0)\pm\dots \pm\mu^m dF_m(0)>0$ is a non-empty open convex cone in which all the differentials are positive, so again
the point $0$ cannot be a local maximum of the function ${\sf F}\left(  x\right)$.

\vskip .2cm
This is exactly what happens for the configuration $C_{\mathfrak{m}}$. The unique linear combination $\lambda(l_1,\dots,l_{12})=0$ of the differentials of the twelve
relevant distances is convex. We
thus have a four-dimensional linear subspace $E$ of the tangent space on which all twelve differentials vanish. Here 4 =
15 (dimension of the configuration space mod $SO(3)$) - 12 (the number of relevant distances) + 1 (the number of relations
between the differentials).

\vskip .2cm
The presence of the linear convex dependency between the differentials is necessary, but not
sufficient, and we have to continue the analysis.
Let $q_1$, $\dots$, $q_{12}$ denote the second differentials of the functions
$F_1$, $\dots$, $F_{12}$. Let $q$ be the restriction of the same convex combination $\lambda$ of the second
differentials to the space $E$, $q=\lambda(q_1,\dots,q_{12})\vert_E$.

\vskip .2cm
Our second observation is that the form $q$
is negatively defined. We prove (and it is not immediate) that the local maximality is implied by these two observations.

\vskip .2cm
Our results imply that the configuration $C_{\mathfrak{m}}$ is unlockable and, moreover, rigid.

\vskip .2cm
The precise meaning of the \textit{unlocking} is the following. Let $\Pi$ be a
collection of non-intersecting open solid bodies, $\Pi=\left\{  \Lambda_{1},\ldots,\Lambda_{k}\right\}  ,$
where each $\Lambda_{i}$ touches the unit central ball, while some distances
between bodies of $\Pi$ are zero. We call a family $\Pi(t)=\left\{  \Lambda_{1}(t),\ldots,\Lambda_{k}(t)\right\}$, $t\geq0$,
of collections of non-intersecting open solid bodies, touching the unit central ball,  a continuous deformation of the collection  $\Pi$ if $\Lambda_j(t)=g_j(t)\Lambda_j$, $j=1,\dots,k$,
where $g_j(t)$ is a continuous curve in the group of Euclidean motions of $\mathbb{R}^3$ with $g_j(0)=\text{Id}$.
 We say that $\Pi$ can be unlocked if there
exists a continuous deformation $\Pi\left(  t\right)$  of $\Pi$  such that some of zero distances between
the members of the configuration $\Pi$ are positive in $\Pi\left(  t\right)  $ for all $t>0$.

\vskip .2cm
We say that a configuration $\Pi$ is {\it rigid} if the only continuous deformations of $\Pi$ are global rotations in the three-dimensional space.

\vskip .2cm
In \cite{K} W. Kuperberg suggested another configuration of six unit non-intersecting cylinders touching the unit sphere and asked whether it can be
unlocked.  It is the
configuration $O_6$ shown on Figure \ref{octahedrConfCyl}.
We are planning to address this question in the forthcoming work \cite{OS-O6}.

\begin{figure}[th]
\vspace{-.4cm} \centering
\includegraphics[scale=0.6]{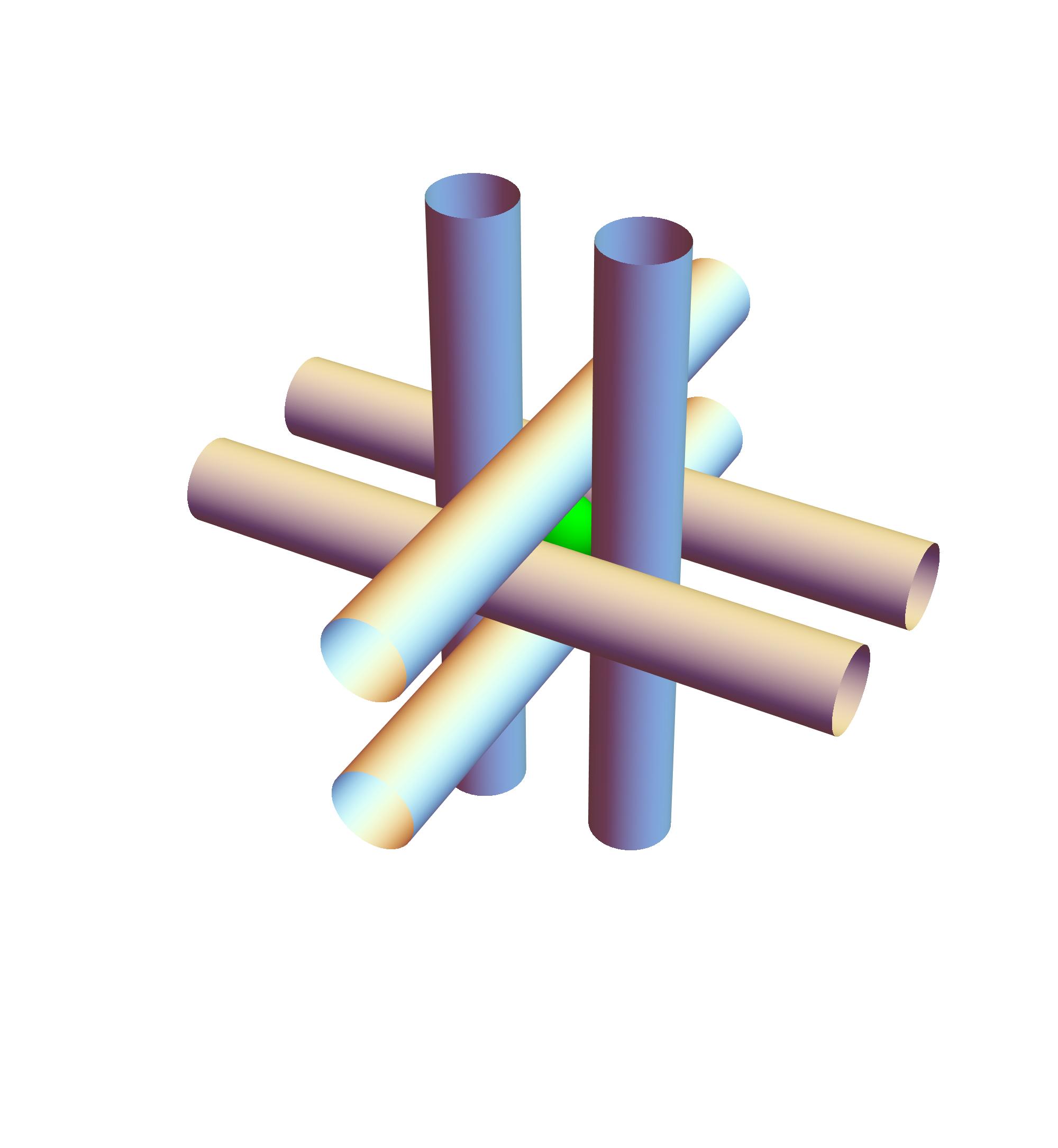}
\caption{Configuration $O_6$ of cylinders}
\label{octahedrConfCyl}
\end{figure}

\vskip .2cm
While calculating variations of distances we observed that the coefficients of the Taylor decompositions of squares of the pairwise distances around
the record point $C_{\mathfrak{m}}$
belong, after a certain normalization, to the field $\mathbb{Q}[\tau]$, where $\tau$ is the golden ratio. This miraculous fact
allows us to reveal a hidden symmetry of the formulas for the coefficients of the Taylor expansions of distances around
the point $C_{\mathfrak{m}}$.
Namely, the Galois conjugation in the field $\mathbb{Q}[\tau]$ restores the original $\mathbb{D}_6$-symmetry of the
configuration $C_6$.

\vskip .2cm
Puzzled by the hidden Galois symmetry, we performed several experiments wondering
whether this symmetry is specific for the record point $C_{\mathfrak{m}}$
or it is inherent at some other points of the curve $C_{6,x}$.
It turns out that for all rational $x$ such that $\sqrt{(1 + x) (1 + 3 x)/3}$ is not rational, we have a similar phenomenon. Namely,
when we perturb the configuration $C_{6,x}$, the Taylor coefficients,  after a certain normalization, belong to a (real) quadratic extension of
the rational field $\mathbb{Q}$, and Galois conjugation restores the  $\mathbb{D}_6$-symmetry of the formulas for variations.

\vskip .4cm
The paper is organized as follows. The next section introduces further
notation, concerning our manifold $M^{6}$. In Section \ref{mairesu} we
formulate our main local maximality results.
Section \ref{seproo} contains the proofs of assertions from Section \ref{mairesu}.
Our general results concerning sufficient conditions for the local extrema of the $min$
function are collected in Section \ref{forms}. The last Section \ref{hiddensymmetry} is devoted to the hidden symmetry of the
configuration $C_{\mathfrak{m}}$
and, more generally, to the hidden symmetry of the configurations $C_{6,x}$.

\section{Configuration manifold}
\label{sectConfManifold}
Here we collect the notation of \cite{OS} used below.

\vskip.2cm
Let $\mathbb{S}^{2}\subset\mathbb{R}^{3}$ be the unit sphere,
centered at the origin. For every $x\in\mathbb{S}^{2}$ we denote by $TL_{x}$
the set of all (unoriented) tangent lines to $\mathbb{S}^{2}$ at $x.$ We
denote by $M$ the manifold of tangent lines to $\mathbb{S}^{2}.$ We represent
a point in $M$ by a pair $\left(  x,\xi\right)  $, where $\xi$ is a unit
tangent vector to $\mathbb{S}^{2}$ at $x,$ though such a pair is not unique:
the pair $\left(  x,-\xi\right)  $ is the same point in $M.$

\vskip .2cm
We shall use the
following coordinates on $M$. Let $\mathbf{x,y,z}$ be the standard coordinate
axes in $\mathbb{R}^{3}$. Let $R_{\mathbf{x}}^{\alpha}$, $R_{\mathbf{y}
}^{\alpha}$ and $R_{\mathbf{z}}^{\alpha}$ be the counterclockwise rotations
about these axes by an angle $\alpha$, viewed from the tips of axes.

\vskip .2cm
We call
the point $\mathsf{N}=\left(  0,0,1\right)  $ the North pole, and
$\mathsf{S}=\left(  0,0,-1\right)  $ -- the South pole. By \textit{meridians}
we mean geodesics on $\mathbb{S}^{2}$ joining the North pole to the South
pole. The meridian in the plane $\mathbf{xz}$ with positive $\mathbf{x}$
coordinates will be called Greenwich. The angle $\varphi$ will denote the
latitude on $\mathbb{S}^{2},$ $\varphi\in\left[  -\frac{\pi}{2},\frac{\pi}
{2}\right]  ,$ and the angle $\varkappa\in\lbrack0,2\pi)$ -- the longitude, so
that Greenwich corresponds to $\varkappa=0.$ Every point $x\in\mathbb{S}^{2}$
can be written as $x=\left(  \varphi_{x},\varkappa_{x}\right)$.

\vskip .2cm
Finally, for each $x\in\mathbb{S}^{2}$, we denote by $R_{x}^{\alpha}$ the rotation by the
angle $\alpha$ about the axis joining $\left(  0,0,0\right)  $ to $x,$
counterclockwise if viewed from its tip, and by $\left(  x,\uparrow\right)  $
we denote the pair $\left(  x,\xi_{x}\right)  ,$ $x\neq\mathsf{N,S,}$ where
the vector $\xi_{x}$ points to the North. We also abbreviate the notation
$\left(  x,R_{x}^{\alpha}\uparrow\right)  $ to $\left(  x,\uparrow_{\alpha
}\right)  $.

\vskip.2cm
Let $u=\left(  x^{\prime},\xi^{\prime}\right)  ,$ $v=\left(
x^{\prime\prime},\xi^{\prime\prime}\right)  $ be two lines in $M$. We denote
by $d_{uv}$ the distance between $u$ and $v$; clearly $d_{uv}=0$ iff $u\cap
v\neq\varnothing.$ If the lines $u,v$ are not parallel then the square of
$d_{uv}$ is given by the formula

\begin{equation}\label{formdist}
d_{uv}^{2}=\frac{\det^{2}[\xi^{\prime},\xi^{\prime\prime},x^{\prime\prime
}-x^{\prime}]}{1-(\xi^{\prime},\xi^{\prime\prime})^{2}}\ ,
\end{equation}
where $(\ast,\ast)$ is the scalar product.

\vskip.2cm
We note that if
$d_{uv}=d>0$ then the cylinders $\text{C}_{u}\left(  r\right)  $ and $\text{C}_{v}\left(
r\right)  ,$ touching $\mathbb{S}^{2}$ at $x^{\prime},x^{\prime\prime},$
having directions $\xi^{\prime},\xi^{\prime\prime},$ and radius $r,$ touch
each other iff
\begin{equation}\label{randd}
r=\frac{d}{2-d}.
\end{equation}
Indeed, if the cylinders touch each other, we have the proportion:
\begin{equation}
\frac{d}{1}=\frac{2r}{1+r}.
\end{equation}

We denote by $M^{6}$ the manifold of 6-tuples
\begin{equation}
\mathbf{m}=\left\{  u_{1},\ldots ,u_{6}:u_{i}\in M,i=1,\ldots ,6\right\}  .
\end{equation}
We are studying the critical points of the function
\[
D\left(  \mathbf{m}\right)  =\min_{1\leq i<j\leq6}d_{u_{i}u_{j}}.
\]
Note that $D\left(  C_{6}\right)  =1$.

\section{The critical point $C_{6}\left(  \varphi_{\mathfrak{m}},\delta_{\mathfrak{m}},\varkappa_{\mathfrak{m}}\right)$\vspace{.2cm}}
\label{mairesu}
The configuration $C_{6}\equiv C_{6}\left(  0,0,0\right)$  in our notation can be
written as

\[
\begin{array}[c]{ll}
C_{6}     = &\left\{   A=\left[ \left(
0,\frac{\pi}{6}\right)   , \uparrow\right]   , D= \left[  \left(  0,\frac{\pi}
{2}\right)   , \uparrow\right]   ,\right.  \\[1.2em]
& \hspace{.6cm}
 B=\left[  \left(  0,\frac{5\pi}{6}\right)
, \uparrow\right]   ,
E=\left[  \left(  0,\frac{7\pi}{6}\right)   , \uparrow
\right]   ,\\[1.2em]
& \hspace{.6cm}\left.
 C=\left[ \left(  0,\frac{3\pi}{2}\right)   , \uparrow\right]   , F=\left[
\left(  0,\frac{11\pi}{6}\right)   , \uparrow\right]   \right\}  .
\end{array}
\]

\vskip .2cm
We need also the configurations $C_{6}\left(  \varphi,\delta,\varkappa\right)$:
\begin{equation}\label{confphideka}
\begin{array}
[c]{ll}
& C_{6}\left(  \varphi,\delta,\varkappa\right)  =\left\{  A=\left[  \left(
\varphi,\frac{\pi}{6}-\varkappa\right)  ,\uparrow_{\delta}\right]  ,D=\left[
\left(  -\varphi,\frac{\pi}{2}+\varkappa\right)  ,\uparrow_{\delta}\right]
,\right.  \\[1em]
& \hspace{2.74cm}B=\left[  \left(  \varphi,\frac{5\pi}{6}-\varkappa\right)
,\uparrow_{\delta}\right]  ,E=\left[  \left(  -\varphi,\frac{7\pi}
{6}+\varkappa\right)  ,\uparrow_{\delta}\right]  ,\\[1em]
& \hspace{2.74cm}\left.  C=\left[  \left(  \varphi,\frac{3\pi}{2}
-\varkappa\right)  ,\uparrow_{\delta}\right]  ,F=\left[  \left(
-\varphi,\frac{11\pi}{6}+\varkappa\right)  ,\uparrow_{\delta}\right]
\right\}  .
\end{array}
\end{equation}

In \cite{OS} we have constructed a a continuous curve
\begin{equation}\label{curgamma}
\gamma(\varphi)=C_{6}\bigl(\varphi,\delta\left(  \varphi\right)
,\varkappa\left(  \varphi\right)  \bigr)\ ,\ \varphi\in\left[  0;\frac{\pi}
{2}\right]\ ,\end{equation}
on which the function $D\bigl(  \gamma(\varphi)\bigr) $
grows for $\varphi\in\left[  0,\varphi_{\mathfrak{m}}\right]  $ and decays for
$\varphi>\varphi_{\mathfrak{m}}$. For the `record' point $C_{\mathfrak{m}}=C_{6}\left(  \varphi_{\mathfrak{m}
},\delta_{\mathfrak{m}},\varkappa_{\mathfrak{m}}\right)$ we have
$$D\bigl(\gamma(\varphi_{\mathfrak{m}})\bigr)=\sqrt{\frac{12}{11}}\ ,$$ with
$$\varphi_{\mathfrak{m}}=\arcsin\sqrt{\frac{3}{11}}\ , \ \varkappa_{\mathfrak{m}
}=-\arctan\frac{1}{\sqrt{15}}\ ,\ \delta_{\mathfrak{m}}=\arctan\sqrt{\frac
{5}{11}} \ .$$
The radii of the corresponding cylinders are equal to $r_{\mathfrak{m}}$ (see formula
(\ref{30})) which, we believe is the maximal possible common radius for six non-inter\-secting cylinders touching the unit ball.

\vskip .2cm
Note that placing the cylinders of radius 1 instead of $r_{\mathfrak{m}}$ leaves
a spacing $2(r_{\mathfrak{m}}-1)$ for each cylinder. Even if we could manage to move these unit cylinders
in such a way that the spacings would behave additively then the total spacing would be
$6\cdot 2(r_{\mathfrak{m}}-1)\approx 1.116843972$ which does not make enough room for a seventh unit cylinder
(the problem of whether seven infinite circular non-intersecting unit cylinders can be arranged about a central unit ball is open).

\vskip .2cm
The configuration $C_{\mathfrak{m}}$ is shown on Figures \ref{record1},
\ref{record2} and \ref{record3}.

\vskip .2cm
There is now an animation, on the page of Yoav Kallus \cite{Ka},
demonstrating the motion of the configuration of 6 cylinders along the curve $\gamma(\varphi)$.

\newpage

\begin{figure}[h!]
\centering
\includegraphics[scale=0.22]{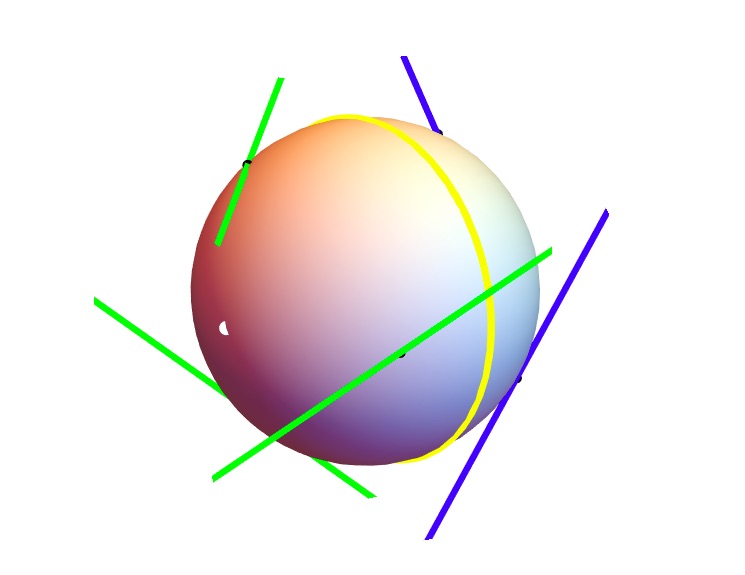}
\caption{Record configuration, side view, the equator is yellow, the north
pole is white\label{record1}}
\includegraphics[scale=0.2]{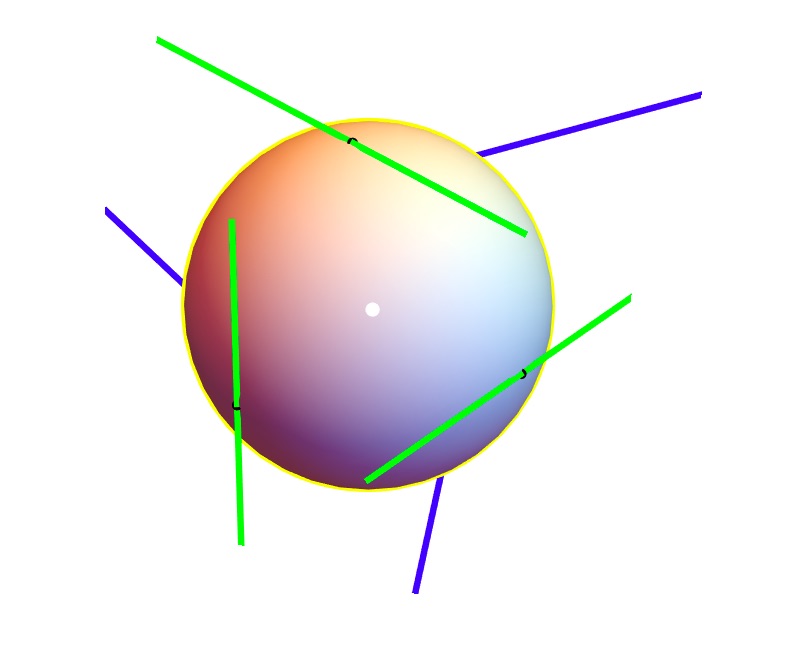}
\caption{ Record configuration again, three upper tangency points shown\label{record2}}
\includegraphics[scale=0.2]{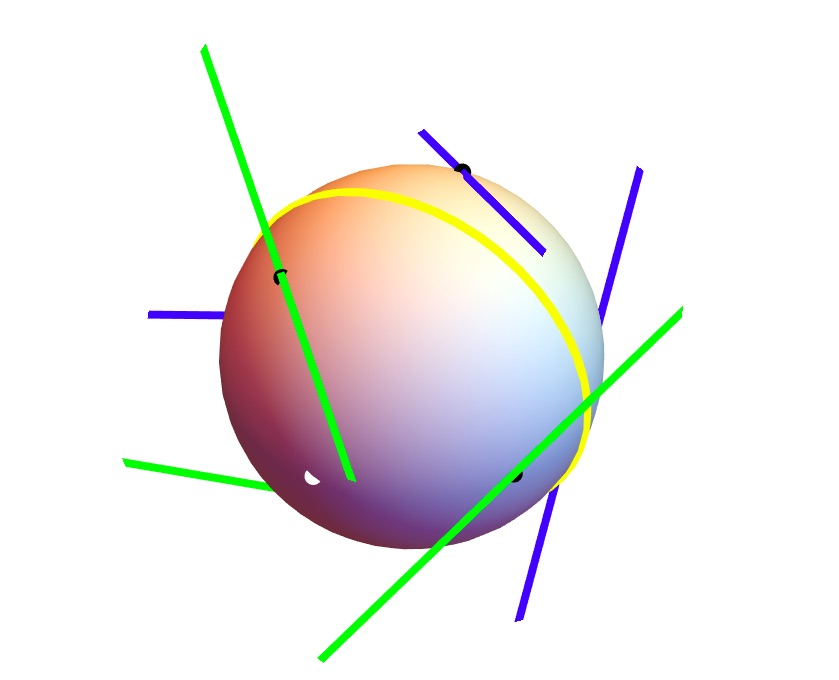}
\caption{Record configuration once more, two upper and one lower tangency points shown\label{record3}}
\end{figure}

\newpage
\subsection{Main maximality result\vspace{.1cm}}\label{secC6}
In \cite{OS} we have shown that the function $D\left(  C_{6}\left(
\varphi,\delta\left(  \varphi\right)  ,\varkappa\left(  \varphi\right)
\right)  \right)$ has a global maximum at the point $\varphi_{\mathfrak{m}}$, corresponding to the configuration
$C_{\mathfrak{m}}$. We now study the function $D$ in the vicinity
of the point $C_{\mathfrak{m}}\ $ in the whole space $M^{6}$.
\begin{theorem}\label{theorC6}
\label{Main} The configuration $C_{\mathfrak{m}}\ $is a point of a sharp local maximum
of the function $D:$ for any point $\mathbf{m}$ in a vicinity of
$C_{\mathfrak{m}}$ we have
\[
D\left(  \mathbf{m}\right)  <\sqrt{\frac{12}{11}}=D\left(  C_{\mathfrak{m}
}\right)  .
\]
\end{theorem}

In the process of the proof, we will see that there exists a 4-dimensional subspace $L_{quadr}$ in the tangent space of
$M^{6}$ at $C_{\mathfrak{m}},$ such that for any $l\in L_{quadr}$ we have
\[ -c_u \left\Vert l \right\Vert t^{2}\leq D\left(  C_{\mathfrak{m}
}+tl\right)  -D\left(  C_{\mathfrak{m}}\right)  \leq -c_d \left\Vert l \right\Vert t^{2}\]
for $t$ small enough. Here $c_d$ and $c_u$ are some constants, $0<c_d \leq c_u <+\infty$ and 
$C_{\mathfrak{m}}+tl\in M^{6}$ stands for the exponential map applied to the tangent vector $tl$.  

\vskip .2cm
For the tangent vectors $l$ outside $L_{quadr}$ we have
\[
-c_u^{\prime}\left(  l\right)  t\leq D\left(  C_{\mathfrak{m}}+tl\right)
-D\left(  C_{\mathfrak{m}}\right)  \leq -c_d^{\prime}\left(  l\right)  t,
\]
where now $c_d^{\prime}\left(  l\right)$ and $c_u^{\prime}\left(  l\right)$ are some positive valued functions of $l$, 
$0<c^{\prime}\left(  l\right)  \leq c^{\prime\prime}\left(  l\right)<+\infty$.

\vskip .2cm
Note, however, that the last two inequalities
do not imply our Theorem, as the example of the Section \ref{toyex} shows.

\subsection{Toy example\vspace{.1cm}}\label{toyex}
Suppose that we know that along any tangent direction to
the critical point the function $D$ decays, either linearly or
at least quadratically. We emphasize that this property by itself
\textit{does not} imply our main result about the local maximality of this point.
An extra work is needed to prove the maximality statements. The following
example explains this.

\vskip .2cm
Let $f$ be a function of two variables defined by
\[
f:=\text{min}\{u_{1},u_{2}\}\ \ \text{where}\ \ u_{1}=-y+3x^{2},u_{2}
=y-x^{2}\ .
\]
The function $f$ equals 0 at the origin.
Consider an arbitrary ray $l$ starting at the origin. Clearly, for some time
this ray evades the `horns' -- the region between the parabolas $y=3x^{2}$ and
$y=x^{2}.$ But outside the horns the function $f$ is negative. Indeed,  inside
the the narrow parabola $y=3x^{2}$ we have $u_{1}<0,u_{2}>0$ so $f$ is
negative there; outside the wide parabola $y=x^{2}$ we have $u_{1}>0,u_{2}<0$
so $f$ is negative there as well. Therefore the origin is a local
maximum of $f$ restricted to $l,$ for any $l.$ Yet the origin is not a local
maximum of the function $f$ on the plane, because inside the horns the
functions $u_{1}$ and $u_{2}$ are positive so $f$ there is positive as well
\footnote{Compare with the `Example of a differentiable function possessing no
extremum at the origin but for which the restriction to an arbitrary line
through the origin has a strict relative minimum there', Chapter 9 in
\cite{GO}.}.

\vskip .2cm Note that there is a convex linear combination of the
differentials of the functions $u_{1}$ and $u_{2}$ which vanishes (the sum of
the differentials); both differentials vanish on the line $y=0$. However the
restriction of the sum of the second differentials to this line is positive, in line with our Theorem \ref{lq2},
Subsection \ref{formsC6}.

\vskip.1cm This toy example captures the essential features of questions we
encounter in the study of local maxima of the distance functions. Some general
theorems needed to establish maximality assertions are formulated and proven in Section \ref{forms}.

\section{Proof of Theorem \ref{theorC6}\vspace{.2cm}}
\label{seproo}

For any $\varphi,\delta,\varkappa$ the configuration $C_{6}\left(  \varphi,\delta,\varkappa\right)$, defined by the
formula (\ref{confphideka}),
possesses the dihedral symmetry group $\mathbb{D}_{3}\equiv\mathbb{Z}_{2}\ltimes\mathbb{Z}_{3}$. The group
$\mathbb{D}_{3}$ is generated by the rotations $R_{\mathbf{z}}^{120^{\circ}}$
and $R_{\mathbf{x}}^{180^{\circ}}$. Because of the $\mathbb{D}_{3}$-symmetry, the fifteen pairwise distances between the lines in the configuration $C_{6}\left(  \varphi,\delta,\varkappa\right)$ split into four groups: we have
$d_{AB}=d_{BC}=d_{CA}=d_{DE}=d_{EF}=d_{FD}$ for the six-plet $\{AB,BC,CA,DE,EF,FD\}$,
and $d_{AD}=d_{BE}=d_{CF}$, $d_{AF}=d_{BD}=d_{CE}$ and $d_{AE}=d_{BF}=d_{CD}$ for the three triplets $\{AD,BE,CF\}$,
$\{AF,BD,CE\}$ and $\{AE,BF,CD\}$. For any configuration $\gamma(\varphi)$ lying on the curve $\gamma$, see (\ref{curgamma}),
twelve of the fifteen distances coincide (we additionally have $d_{AB}=d_{AD}=d_{AF}$).
In the proof we shall study variations of distances in a vicinity of the point
$C_{6}\left(  \varphi_{\mathfrak{m}},\delta_{\mathfrak{m}},\varkappa_{\mathfrak{m}}\right)\equiv\gamma(\varphi_{\mathfrak{m}})$ in $M^{6}$. The distances from the $\{AE,BF,CD\}$-triplet are greater than the other twelve (equal) distances for the
configuration $C_{6}\left(  \varphi_{\mathfrak{m}},\delta_{\mathfrak{m}},\varkappa_{\mathfrak{m}}\right)$, see \cite{OS}, so
these three distances are not relevant in the study of the local maximality of the configuration
$C_{6}\left(  \varphi_{\mathfrak{m}},\delta_{\mathfrak{m}},\varkappa_{\mathfrak{m}}\right)$.

\vskip .2cm
The perturbed position of a tangent line $J\in \{A,B,C,D,E,F\}$ is
$$J=J(\varkappa_{\mathfrak{m}}+\Delta_J^{\varkappa},\varphi_{\mathfrak{m}}+
\Delta_J^{\varphi},\delta_{\mathfrak{m}}+\Delta_J^{\delta})$$
where
$$\{\varkappa_{\mathfrak{m}},\varphi_{\mathfrak{m}},\delta_{\mathfrak{m}}\}$$
is the maximal point on our $D_3$-symmetric curve and
\begin{equation}\label{strnorma}
\begin{array}{c}
\displaystyle{ \Delta_J^{\varkappa}=\frac{11}{32\sqrt{3}}\left(J_{\varkappa,1}t+J_{\varkappa,2}t^2+o(t^2)\right)\ ,}\\[1.5em]
\displaystyle{ \Delta_J^{\varphi}=\frac{11}{4\sqrt{6}}\left(J_{\varphi,1}t+J_{\varphi,2}t^2+o(t^2)\right)\ ,}\\[1.5em]
\displaystyle{ \Delta_J^{\delta}=\frac{11\sqrt{11}}{48}\left(J_{\delta,1}t+J_{\delta,2}t^2+o(t^2)\right)\ .}
\end{array}\end{equation}
The numerical factors here are introduced for convenience, because with this normalization the coefficients of the Taylor decompositions  of squares of distances around the record point
$C_{6}\left(  \varphi_{\mathfrak{m}},\delta_{\mathfrak{m}},\varkappa_{\mathfrak{m}}\right)$ belong
(we do not understand the reason for this) to the field $\mathbb{Q}[\tau]$, where $\tau$ is the golden section,
$$\tau=\frac{1+\sqrt{5}}{2}\ ,\ \bar\tau=\frac{1-\sqrt{5}}{2}\ .$$

\vskip .2cm
To fix the rotational symmetry we keep the tangent line $A$ at its place, that is,
$A_{\varphi,j}=A_{\varkappa,j}=A_{\delta,j}=0$, $j=1,2,\dots$, working therefore in the 15-dimensional space
$M^{6}$ mod $SO(3)$.

\subsection{Variation of distances in the first order\vspace{.1cm}}\label{vadifi}
We first list all differentials.

\vskip .2cm
Let
\begin{equation}\label{begade}
\begin{array}{c}
\displaystyle{ \beta_{11}=\frac{2}{11}(4-\sqrt{5})\ ,\ \bar\beta_{11}=\frac{2}{11}(4+\sqrt{5})\ ,}\\[1em]
\displaystyle{ \gamma_{-19}=-1-2\sqrt{5}\ ,\ \bar\gamma_{-19}=-1+2\sqrt{5}\ .}\end{array}\end{equation}
The seemingly strange indices of $\beta$ and $\gamma$ are
used due to the prime decompositions
$$11=(4+\sqrt{5})(4-\sqrt{5})\ ,\ -19=(-1-2\sqrt{5})(-1+2\sqrt{5})$$
in the ring $\mathbb{Z}[\tau]$ of golden integers.

\vskip .2cm
The differentials of the $\{AF,CE,BD\}$-triplet are
$$\left[d(AF)^2\right]_1=\beta_{11}\left( \tau^2\beta_{11} F_{\varkappa,1}+\tau^3F_{\varphi,1}+F_{\delta,1}\right),$$
$$\left[d(CE)^2\right]_1=\beta_{11}\left( \tau^2\beta_{11} ( C_{\varkappa,1}+E_{\varkappa,1})+\tau^3(C_{\varphi,1}+E_{\varphi,1})+
( C_{\delta,1}+E_{\delta,1})\right),$$
$$\left[d(BD)^2\right]_1=\beta_{11}\left( \tau^2\beta_{11} ( B_{\varkappa,1}+D_{\varkappa,1})+\tau^3(B_{\varphi,1}+D_{\varphi,1})+
( B_{\delta,1}+D_{\delta,1})\right).$$

\vskip .2cm
The differentials of the $\{CF,BE,AD\}$-triplet are
$$\left[d(CF)^2\right]_1=-\bar\beta_{11}\left( \bar\tau^2\bar\beta_{11} ( C_{\varkappa,1}+F_{\varkappa,1})-\bar\tau^3 (C_{\varphi,1}+F_{\varphi,1})+( C_{\delta,1}+F_{\delta,1})\right),$$
$$\left[d(BE)^2\right]_1=-\bar\beta_{11}\left( \bar\tau^2\bar\beta_{11} ( B_{\varkappa,1}+E_{\varkappa,1})-\bar\tau^3 (B_{\varphi,1}+E_{\varphi,1})+( B_{\delta,1}+E_{\delta,1})\right),$$
$$\left[d(AD)^2\right]_1=-\bar\beta_{11}\left( \bar\tau^2\bar\beta_{11} D_{\varkappa,1}-\bar\tau^3 D_{\varphi,1}+D_{\delta,1}\right).$$
\vskip .2cm
For the 6-plet $\{AB,BC,CA,DE,EF,FD\}$ we have
$$\left[d(AB)^2\right]_1=\frac{1}{5}\left(  B_{\varkappa,1}-2\tau\bar\gamma_{-19}B_{\varphi,1}-2\bar\tau B_{\delta,1}\right),$$
$$\left[d(BC)^2\right]_1=\frac{1}{5}\left(  C_{\varkappa,1}- B_{\varkappa,1}-2\bar\tau\gamma_{-19}B_{\varphi,1}-2\tau\bar\gamma_{-19}C_{\varphi,1}+
2\tau B_{\delta,1}-2\bar\tau C_{\delta,1}\right),$$
$$\left[d(CA)^2\right]_1=\frac{1}{5}\left( -C_{\varkappa,1} -2\bar\tau\gamma_{-19}C_{\varphi,1}+2\tau C_{\delta,1}\right),$$
$$\left[d(DE)^2\right]_1=\frac{1}{5}\left(  D_{\varkappa,1}- E_{\varkappa,1}-2\bar\tau\gamma_{-19}E_{\varphi,1}-2\tau\bar\gamma_{-19}D_{\varphi,1}+
2\tau E_{\delta,1}-2\bar\tau D_{\delta,1}\right),$$
$$\left[d(EF)^2\right]_1=\frac{1}{5}\left(  E_{\varkappa,1}- F_{\varkappa,1}-2\bar\tau\gamma_{-19}F_{\varphi,1}-2\tau\bar\gamma_{-19}E_{\varphi,1}+
2\tau F_{\delta,1}-2\bar\tau E_{\delta,1}\right),$$
$$\left[d(FD)^2\right]_1=\frac{1}{5}\left(  F_{\varkappa,1}- D_{\varkappa,1}-2\bar\tau\gamma_{-19}D_{\varphi,1}-2\tau\bar\gamma_{-19}F_{\varphi,1}+
2\tau D_{\delta,1}-2\bar\tau F_{\delta,1}\right).$$
The expressions for the differentials $\left[d(AF)^2\right]_1$, $\left[d(AD)^2\right]_1$, $\left[d(AB)^2\right]_1$ and
$\left[d(CA)^2\right]_1$ look shorter but this is only because of our convention to keep the tangent line $A$ fixed.

\subsection{Linear dependence of differentials\vspace{.1cm}}
Our 12 linear functionals are not linearly independent. To see this consider
the functionals
$$\mathcal{S}_1\!:=\!\left[d(AB)^2\right]_1\!+\!\left[d(BC)^2\right]_1\!+\!\left[d(CA)^2\right]_1\!+\!\left[d(DE)^2\right]_1\!+\!\left[d(EF)^2\right]_1\!+\!\left[d(FD)^2\right]_1,$$
$$\mathcal{S}_2\!:=\!\left[d(AF)^2\right]_1\!+\!\left[d(CE)^2\right]_1\!+\!\left[d(BD)^2\right]_1$$
and
$$\mathcal{S}_3\!:=\!\left[d(CF)^2\right]_1\!+\!\left[d(BE)^2\right]_1\!+\!\left[d(AD)^2\right]_1.$$
Let also (keeping in mind that $A_{\varphi,1}=A_{\varkappa,1}=A_{\delta,1}=0$)
$$\mathcal{F}:=\left( A_{\varphi,1}+B_{\varphi,1}+C_{\varphi,1}+D_{\varphi,1}+E_{\varphi,1}+F_{\varphi,1}\right)$$
and
$$\mathcal{D}:=\left( A_{\delta,1}+B_{\delta,1}+C_{\delta,1}+D_{\delta,1}+
E_{\delta,1}+F_{\delta,1}\right).$$

A direct computation shows that
$$\mathcal{S}_1=
-\frac{18}{5}\mathcal{F}+\frac{2}{\sqrt{5}}\mathcal{D}\ ,$$
and
$$(23+3\sqrt{5})\mathcal{S}_2+(23-3\sqrt{5})\mathcal{S}_3=36\mathcal{F}-4\sqrt{5}\mathcal{D}\ ,$$
so the following strictly convex combination (that is, the linear combination with positive coefficients) of the differentials vanish:
\begin{equation}\label{concomconfC}
10\mathcal{S}_1+(23+3\sqrt{5})\mathcal{S}_2+(23-3\sqrt{5})\mathcal{S}_3=0\ .\end{equation}

This is the only relation between the differentials: a direct computation shows that the linear space, spanned by the differentials, is 11-dimensional.

\vskip .2cm
Let $E$ denote the null-space of our differentials, i.e. the linear subspace of the tangent space on which all the differentials vanish. It has the dimension $4=15-11$.

\subsection{Relevant quadratic form\vspace{.1cm}}
We have now to study our functions to the next order, $t^{2}$.
We have performed the calculation of the 12 quadratic forms using Mathematica \cite{W}.
The formulas are quite lengthy and we do not reproduce the full details
since it is just an intermediate result.

\vskip .2cm
The only quadratic form important to us is, as we will explain in Section \ref{forms}, the same combination (\ref{concomconfC}) but calculated
for $\left[d(**)^2\right]_2$ instead of $\left[d(**)^2\right]_1$.

\vskip .2cm
This is a quadratic form in 15 variables. However, due to the results of Section \ref{forms}, it is enough to
calculate the restriction of this form to the null-space $E$ of the differentials.
If this restriction to $E$ would be  negatively defined, that will prove our result.

\vskip .2cm
The subspace $E$ has dimension 4. As independent variables on $E$ we choose
$$w_1:=E_{\varkappa,1}\ ,\ w_2:=E_{\varphi,1}\ ,\ w_3:=B_{\delta,1}\ \text{ and }\  w_4:=C_{\delta,1}\ .$$ The resulting quadratic form on $E$ is $\sum \Phi_{ij}w_i w_j$ where

\begin{equation}\label{negPhi}
\Phi =\frac{11}{9}\left(\begin{array}{cccc}
\displaystyle{ -\frac{919}{24}}&\displaystyle{ \frac{5663}{12\sqrt{5}}}&\displaystyle{ -\frac{\bar\mu_1}{30} }
&\displaystyle{ -\frac{\mu_1}{30} }\\[1.5em]
\displaystyle{ \frac{5663}{12\sqrt{5}} }&\displaystyle{ -\frac{18663}{6}}&\displaystyle{ -\frac{7\bar\mu_2}{15}}
&\displaystyle{ \frac{7\mu_2}{15}}\\[1.5em]
\displaystyle{ -\frac{\bar\mu_1}{30}}&\displaystyle{ -\frac{7\bar\mu_2}{15}}&\displaystyle{ -\frac{4\bar\mu_3}{15}}&700\\[1.5em]
\displaystyle{ -\frac{\mu_1}{30}}&\displaystyle{ \frac{7\mu_2}{15}}&700&\displaystyle{ -\frac{4\mu_3}{15}}
\end{array}\right)\ .\end{equation}\vskip .2cm
\noindent Here $\mu_1=2865+1438\sqrt{5}$, $\mu_2=3530+939\sqrt{5}$, $\mu_3=5335+1878\sqrt{5}$, and bar stands for the Galois conjugation (the replacement of $\sqrt{5}$ by $-\sqrt{5}$) in the field $\mathbb{Q}[\tau]$.

\vskip .2cm
A direct calculation shows that $\Phi$ is indeed negatively defined and the result follows, by applying the theorem
\ref{lq2}, Section \ref{forms}.

\filbreak
\section{Local maximum\vspace{.2cm}}\label{forms}
This section provides a sufficient condition which ensures that the point $\mathbf{0}\in\mathbb{R}^{n}$ is a sharp local maximum
of the function
\begin{equation} {\sf F}\left(  x\right)  :=\min\left\{  F_{1}\left(  x\right)
,\dots,F_{m}\left(x\right)  \right\} \ ,\label{deffuF}\end{equation}
where the functions $F_{1}\left(  x\right)  ,\dots,F_{m}\left(  x\right)$ are analytic in a neighborhood of $0\in\mathbb{R}^n$
and $F_u(0)=0$,  $u=1,\dots,m$
(for the configurations of tangent lines in Theorem \ref{Main}
the functions $F_u$ are the differences between the squares of distances in the perturbed and initial configurations). This sufficient condition is needed
to complete the proofs of Theorem \ref{Main}.

\vskip .2cm
We denote by $l_{uj}$ and $q_{ujk}$ the coefficients of the linear and quadratic parts of the function $F_u(x)$, $u=1,\dots,m$,
\begin{equation}\label{decofuf}F_u(x)=l_{uj}x^j+q_{ujk}x^j x^k+o(2)\ ,\end{equation}
where $o(2)$ stand for higher order terms. Here and till the end of the Section the summation over repeated coordinate indices is assumed.

\vskip .2cm
Let $\xi^j$, $j=1,\dots,n$, be the coordinates, corresponding to the coordinate system $x^1,\dots,x^n$, in the tangent space
to  $\mathbb{R}^n$ at the origin. We define the linear and quadratic forms $l_u(\xi)\equiv l_{uj}\xi^j$ and $q_u(\xi)\equiv q_{ujk}\xi^j\xi^k$ on the
tangent space  $T_0\mathbb{R}^n$. Let $E$ be the subspace in $T_0\mathbb{R}^n$ defined as the intersection of kernels of the linear forms $l_u(\xi)$,
$$E=\bigcap_{u=1}^m\; \ker l_u(\xi)\ .$$

The configuration $C_{6}\left(  \varphi_{\mathfrak{m}}
,\delta_{\mathfrak{m}},\varkappa_{\mathfrak{m}}\right)$
provides a particular example, see Section \ref{seproo}, of a family $\left\{  F_{1}\left(  x\right)
,\dots,F_{m}\left(  x\right)  \right\}$ of $m$ analytic functions in $n$ variables, $m\leq n$, possessing the following two properties:
\begin{itemize}
\item[(A)] The linear space, generated by the linear forms $l_{1}(\xi),\dots,l_{m}(\xi)$ is $(m-1)$-dimensional; moreover, the linear relation between
$l_{1}(\xi),\dots,l_{m}(\xi)$ is strictly convex,
\begin{equation}\lambda^{1}l_{1}(\xi)+\ldots
+\lambda^m l_{m}(\xi)=0\ ,\label{32}\end{equation}
with $\lambda^{i}>0\ ,\ 1\leq i\leq m-1$.
Therefore,
\begin{equation}
\label{mima1}\text{If\ $l_u(\xi)\geq 0$\ for\ all\ $u=1,\dots,m$\ then\ $\xi\in E$}\ .
\end{equation}

\item[(B)] The inequality
\begin{equation}\label{uslo2}
\bigl(\lambda^{1}q_{1}(\xi)+\ldots
+\lambda^m q_{m}(\xi)\bigr)\vert_{_{\xi\in E}}\geq 0
\end{equation}
admits only the trivial solution $\xi=0$.
\end{itemize}

In the setting of Section \ref{seproo}, $n=15$ and $m=12$.

\vskip .2cm
For the configuration $C_{6}\left(  \varphi_{\mathfrak{m}},\delta_{\mathfrak{m}},\varkappa_{\mathfrak{m}}\right)$ the properties (A) and (B) hold. The only relation
between the differentials is the relation (\ref{concomconfC}); this is the property (A). The property (B) follows since the form (\ref{negPhi}) is negatively defined.

\vskip .2cm
We note that if the linear space, spanned by the linear forms $l_{u}(\xi)$, $u=1,\dots,m$, is $(m-1)$-dimensional
then the implication (\ref{mima1}) is satisfied if and only if the unique linear dependency between the differentials is strictly convex.

\begin{theorem}\label{lq2}
Under the conditions {\rm (A)} and {\rm (B)}, the origin is the strict local maximum of the function ${\sf F}(x)$. \end{theorem}

Below we give two proofs of Theorem \ref{lq2}. These proofs are different in nature. We present both of them because of the importance of Theorem
\ref{lq2} for our conclusion about the local maximality of the configuration $C_{6}\left(  \varphi_{\mathfrak{m}},\delta_{\mathfrak{m}},\varkappa_{\mathfrak{m}}\right)$.

\vskip .2cm
Although Theorem \ref{lq2} is formulated for analytic functions, any of the two proofs show that
Theorem \ref{lq2} holds in fact for functions $F_j$ of the class $\mathcal{C}^{3}$.

\subsection{First proof of Theorem \ref{lq2}\vspace{.1cm}}\label{formsC6}

\begin{lemma}\label{inchava}
The conditions {\rm (A)} and {\rm (B)} are invariant under an arbitrary analytic change of variables, preserving the origin,
\begin{equation}\label{chavar}x^j=A^j_k\tilde{x}^k+A^j_{kl}\tilde{x}^k\tilde{x}^l+o(2)\ .\end{equation}
Here the matrix $A^j_k$ is non-degenerate.\end{lemma}

\begin{proof} Substituting (\ref{chavar}) into the decompositions (\ref{decofuf}) we find
$$\begin{array}{rcl}F_u(x)&=&
l_{uj}\left(A^j_k\tilde{x}^k+A^j_{kl}\tilde{x}^k\tilde{x}^l\right)+q_{uij}A^i_kA^j_l\tilde{x}^k\tilde{x}^l+o(2)\\[1em]
&=&\tilde{l}_{uk}\tilde{x}^k+\tilde{q}_{ukl}\tilde{x}^k \tilde{x}^l+o(2)\ ,
\end{array}$$
where
$$\tilde{l}_{uj}=l_{uj}A^j_k\ \ \text{and}\ \ \tilde{q}_{ukl}=l_{uj}A^j_{kl}+q_{uij}A^i_kA^j_l\ .$$
The assertion is immediate for the condition (A).
As for the condition (B), it is enough to note that $\lambda^1 l_{1j}+\dots+\lambda^m l_{mj}=0$, $j=1,\dots,n$.
\end{proof}

\vskip .2cm
\noindent First {\bf Proof} of Theorem \ref{lq2}. It follows from the condition (A) that the differentials of the functions $F_1(x),\dots,F_{m-1}(x)$ are
independent; by using the implicit function theorem we change the variables to have
$$F_1(x)=x_1\ ,\ \dots \ ,\ F_{m-1}(x)=x_{m-1}\ .$$
To make the notation lighter we use the same letters $x_j$ instead of $\tilde{x}_j$.

\vskip .2cm
We identify the tangent subspace $E\subset T_0\mathbb{R}^n$ with the
plane $x_1=\ldots=x_{m-1}=0$.

\vskip .2cm
The remaining function $F_{m}(x)$ is
$$F_{m}(x)=-\,\frac{1}{\lambda^m}\,\sum_{i=1}^{m-1}\lambda^i x_i+q(x)+o(2)\ ,$$
where $q(x)$ is a quadratic form. Let
$${\sf q}(x_m,\ldots ,x_n):=q(x)\,\rule[-.22cm]{0.1mm}{.54cm}_{\; x_1=\ldots=x_{m-1}=0}$$
be the restriction of the quadratic form $q(x)$ to $E$.

\vskip .2cm
We consider separately two cases: (i) $x\in E$ and (ii) $x\notin E$.

\vskip .2cm
(i) The property (B) refers in our situation to the quadratic form ${\sf q}$.
Due to Lemma \ref{inchava}, the form ${\sf q}$ is negatively defined on $E$. Thus there exists a small neighborhood $V$ of the origin in $E$ such that
if $x\in V\setminus \{0\}$ then $F_m(x)<0$, so ${\sf F}\left(  x\right)  =\min\left\{  0,F_{m}\left(x\right)  \right\}<0$.

\vskip .2cm
(ii) If $U$ is a small enough neighborhood of the origin in $\mathbb{R}^n$ (in the chosen coordinate system the smallness depends only on the coefficients of the function $F_m(x)$)
and $x\notin U\cap E$ then the property (\ref{mima1}) implies that there exists $j$, $1\leq j\leq m$, such that $l_j(x)<0$. We have two subcases.

\vskip .2cm
(ii)$_1$ If there is a $j$, $1\leq j< m$, such that $l_j(x)<0$ then $F_j(x)=l_j(x)<0$ hence ${\sf F}\left(  x\right)<0$.

\vskip .2cm
(ii)$_2$ Otherwise, we have that all $x_i\geq 0$, $i=1,\dots, m-1$. We set
$x_i=z_i^2$, $i=1,\ldots , m-1$. In terms of the variables $z_1,\dots,z_{m-1},x_m,\dots$, the function $F_m(x)$ has the form
$$F_m(z_1^2,...\, ,z_{m-1}^2,x_m,... )=-\frac{1}{\lambda^m}\sum_{i=1}^{m-1}\lambda^i z_i^2+{\sf q}(x_m,...\, ,x_n)+\text{higher order terms}\ .$$
The quadratic form $-\frac{1}{\lambda^m}\sum_{i=1}^{m-1}\lambda^i z_i^2+{\sf q}(x_m,\ldots ,x_n)$ is negatively defined, so the function $F_m(z_1^2,\dots,z_{m-1}^2,x_m,\dots)$ is strictly less than zero in a
punctured neighborhood of the origin. This implies that the function $F_m(x)$ is strictly negative whenever all $x_i\geq 0$, $i=1,\dots, m-1$, and at least
one of them is positive. Thus again ${\sf F}\left(  x\right)<0$. \hspace{.15cm}\myblacksquare
\paragraph{Remark.} It follows from the proof that the function ${\sf F}(x)$ decays quadratically at zero along
any direction in $E$ and decays linearly along any direction outside $E$.

\subsection{Second proof of Theorem \ref{lq2}\vspace{.15cm}}\label{C6anopr}
The cornerstone  of the second proof is the set 
\begin{equation}
\mathcal{E}=\left\{  x\in\mathbb{R}^{n}:F_{1}\left(  x\right)  =\ldots
=F_{m}\left(  x\right)  \right\}  \label{91}\ .
\end{equation}

We assume that all occuring real vector spaces are equipped with a Euclidean structure. For a vector $v$ we denote
by $\hat{v}$ the unit vector in the direction of the vector $v$.

\vskip .2cm
Our proof will use the following observation.

\begin{lemma}\label{lemobse}
Let $\underline{\lambda}=\{\lambda^{1},\ldots ,\lambda^{m}\}$ be a collection of $m$ positive real numbers, $\lambda^{j}>0$, $j=1,\dots,m$.
Let $\mathcal{W}_{\underline{\lambda}}$
be the space of $m$-tuples $\{ v_{1},\ldots ,v_{m}\}$ of
vectors in $\mathbb{R}^{m-1}$, generating the space $\mathbb{R}^{m-1}$ and such that
\begin{equation}\lambda^{1}v_{1}+\ldots +\lambda^{m}v_{m}=0\ .\label{colire}\end{equation}
Then there exists a continuous positive-valued function $\delta\colon \mathcal{W}_{\underline{\lambda}}\to\mathbb{R}_{>0}$
such that for any unit vector ${\sf s}\in\mathbb{R}^{m-1}$  we have
\begin{equation}
\min_{i}\left\langle {\sf s},\hat{v}_{i}\right\rangle <-\delta\left(  v_{1},\dots ,v_{m}
\right)  .\label{111}\end{equation}
\end{lemma}
\begin{proof}
For an angle $\alpha$, $0\leq\alpha <\pi$, let $D_{j}\left(  \alpha\right)$, $j=1,\dots ,m$, denote the
open spherical cap, centered at ($-\hat{v}_{j}$), on the unit
sphere $\mathbb{S}^{m-2}$, consisting of all the points ${\sf s}\in\mathbb{S}^{m-2}$
such that the angle $\measuredangle\left(  {\sf s},\hat{v}_{j}\right)  >\alpha$.

\vskip .2cm
For any unit vector ${\sf s}$ there exists an index $i$ such that $\left\langle {\sf s},v_{i}\right\rangle <0$.
Indeed, since the vectors $v_{1},\ldots ,v_{m}$ span the whole space $\mathbb{R}^{m-1}$,
some of the scalar products $\left\langle {\sf s},v_{j}\right\rangle$, $j=1,\dots ,m$, are nonzero.
Taking the scalar product of the relation (\ref{colire}) with the
vector ${\sf s}$ we see that at least one of the scalar products $\left\langle {\sf s},v_{i}\right\rangle $
has to be negative. Therefore
\[ \bigcup_{i=1}^{m}\, D_{i}\left(  \frac{\pi}{2}\right)  =\mathbb{S}^{m-2}\ .\]
Thus,
$$\alpha_{0}\left(  v_{1},\ldots,v_{m}\right)  >\frac{\pi}{2}\ ,$$
where the function $\alpha_{0}\left(  v_{1},\ldots,v_{m}\right)$ is defined by
\[\alpha_{0}\left(  v_{1},\ldots,v_{m}\right)  =\sup\left\{  \alpha
:\bigcup_{i=1}^{m}D_{i}\left(  \alpha\right)  =\mathbb{S}^{m-2}\right\}\ .\]
Let
\[\bar{\alpha}\left(  v_{1},\ldots,v_{m}\right)  :=\frac{1}{2}\left[  \alpha
_{0}\left(  v_{1},\ldots,v_{m}\right)  +\frac{\pi}{2}\right]  \ .\]
Clearly, $\bigcup_{i=1}^{m} D_{i}\left(  \bar{\alpha}\right)  =\mathbb{S}^{m-2}$.
Define the function $\delta$ by
\[
\delta\left(  v_{1},\ldots,v_{m}\right)  =-\cos\bar{\alpha}\left(
v_{1},\ldots,v_{m}\right)  .
\]
With this choice of the function $\delta$ the relation $\left(  \ref{111}\right)  $ clearly
holds. The positivity and the continuity of the function $\delta$ are straightforward.
\end{proof}

\vskip .2cm
We return to the consideration of our analytic functions.

\begin{lemma}
\label{prele} If the point $y\in\mathbb{R}^{n}$ happens to be away
from the set $\mathcal{E}$, see (\ref{91}), 
and the norm $\left\Vert y\right\Vert $ is small enough then one can find a
point $x$ on $\mathcal{E}$ such that $\mathsf{F}\left(  y\right)
<\mathsf{F}\left(  x\right)  .$

\vskip .2cm
Moreover, there exists a constant $c$ such that for $y\notin\mathcal{E},$ and
$x=x\left(  y\right)  \in\mathcal{E}$ being the point in $\mathcal{E}$ closest to $y$ we have
\begin{equation}
\mathsf{F}\left(  y\right)  <\mathsf{F}\left(  x\right)  -c\left\Vert
x-y\right\Vert ,\label{85}
\end{equation}
provided, again, that  the norm $\left\Vert y\right\Vert $ is small enough.
\end{lemma}

\begin{proof}
Since there is only one linear dependency between the differentials $l_{1},\dots,l_{m}$ of the
functions $F_{1}(x),\dots,F_{m}(x)$, the set $\mathcal{E}$ is a smooth
manifold in a vicinity of the origin, of dimension $n-m+1$.

\vskip .2cm  We introduce the tubular neighborhood $U_{r}\left(
\mathcal{E}\right)  $ of the manifold $\mathcal{E}$, which is comprised by all
points $y$ of $\mathbb{R}^{n}$ which can be represented as $\left(
x,\mathsf{s}_{x}\right)  ,$ where $x\in\mathcal{E}$ and $\mathsf{s}_{x}$ is a
vector normal to $\mathcal{E}$ at $x,$ with norm less than $r.$  Let
$\mathcal{E}_{r^{\prime}}\subset\mathcal{E}$ be the neighborhood of the origin
in $\mathcal{E}$, comprised by all $x\in\mathcal{E}$ with norm $\left\Vert
x\right\Vert <r^{\prime},$ and $U_{r}\left(  \mathcal{E}_{r^{\prime}}\right)
$ be the part of $U_{r}\left(  \mathcal{E}\right)  $ formed by points hanging
over $\mathcal{E}_{r^{\prime}}.$ If both $r$ and $r^{\prime}$ are small enough
then every $y\in U_{r}\left(  \mathcal{E}_{r^{\prime}}\right)  $ can be
written as $\left(  x,\mathsf{s}_{x}\right)  $ with $x\in\mathcal{E}
_{r^{\prime}}$ in a unique way. Note that $x$ is the point
on $\mathcal{E}$ closest to $y$. Also, for any $r,r^{\prime}>0$ the set
$U_{r}\left(  \mathcal{E}_{r^{\prime}}\right)  $ evidently contains an open
neighborhood of the origin.

\vskip .2cm Now we are going to show that if $y=\left(  x,\mathsf{s}
_{x}\right)  \in U_{r}\left(  \mathcal{E}_{r^{\prime}}\right)  ,$
$\mathsf{s}_{x}\neq0,$ and both $r$ and $r^{\prime}$ are small enough then
$\mathsf{F}\left(  y\right)  <\mathsf{F}\left(  x\right) $. To this end, let
$N_{x}$ be the plane normal to $\mathcal{E}$ at $x$ (so that $\mathsf{s}
_{x}\in N_{x}$). We identify $N_{x}$ with the linear space $\mathbb{R}^{m-1},$
so that $x$ corresponds to $0\in\mathbb{R}^{m-1}$.

\vskip .2cm Now we will use Lemma \ref{lemobse}, applied not to a single
space, but to the whole collection of the $\left(  m-1\right) $-dimensional
spaces $N_{x},$ $x\in\mathcal{E}_{r^{\prime}}.$ To do this, we equip each
$N_{x}$ with $m$ vectors $v_{1}^{x},\ldots,v_{m}^{x}\in N_{x},$ which generate
$N_{x}$ and which satisfy the same convex linear relation. All this data is
readily supplied by the linear functionals $l_{1},\ldots,l_{m},$ restricted to
$N_{x}.$ Indeed, each restricted functional $l_{j}^{x}\equiv l_{j}{ |}_{N_{x}
}$ can be uniquely written as $l_{j}^{x}\left(  \ast\right)  =\left\langle
\ast,v_{j}^{x}\right\rangle ,$ with $v_{j}^{x}\in N_{x}.$ Here the scalar
product on $N_{x}$ is the one restricted from $\mathbb{R}^{n}.$ Clearly, for
every $x$ we have
\[
\lambda^{1}v_{1}^{x}+\ldots+\lambda^{m}v_{m}^{x}=0\ ,
\]
since for every vector $\mathsf{s}\in N_{x}$ we have $\lambda^{1}l_{1}\left(
\mathsf{s}\right)  +\ldots+\lambda^{m}l_{m}\left(  \mathsf{s}\right)  =0$ (as
for any other vector). Moreover, $l_{j}\left(  \mathsf{s}\right) <0$ for some
$j=j(\mathsf{s})$, $1\leq j\leq m$, see formula (\ref{mima1}) or the proof of
Lemma \ref{lemobse}.

\vskip.2cm Since the space $N_{x=0}$ is orthogonal to the null-space $E$, the
$m$ vectors $v_{1}^{0},\ldots,v_{m}^{0}$ do generate $N_{0}$. Because the
spaces $N_{x}$ depend on $x$ continuously, all of them are transversal to $E$,
provided $r^{\prime}$ is small. Thus, the vectors $v_{1}^{x},\ldots,v_{m}^{x}$
do generate the spaces $N_{x}$ for all $x\in\mathcal{E}_{r^{\prime}},$
provided again that $r^{\prime}$ is small enough. Lemma \ref{lemobse} provides
us now with a collection of functions $\delta^{x}$ on the spaces
$\mathcal{W}_{\underline{\lambda}}^{x}$ of $m$-tuples of vectors from $N_{x}.$
It follows from the continuity, in $x$, of the spaces $N_{x}$ and the
$m$-tuples $\{v_{1}^{x},\ldots,v_{m}^{x}\}$, and from the Lemma \ref{lemobse}
that the functions $\delta^{x}$ can be chosen in such a way that the resulting
positive function $\Delta(x):=\delta^{x}\left(  v_{1}^{x},\ldots,v_{m}
^{x}\right)  $ on $\mathcal{E}_{r^{\prime}}$ is continuous in $x$ and also is
uniformly positive, that is,
\[
\Delta(x)>2c\text{ for all }x\in\mathcal{E}_{r^{\prime}}\ ,
\]
for some $c>0$, provided $r^{\prime}$ is small enough.

\vskip.2cm In virtue of Lemma \ref{lemobse}, for every $x\in\mathcal{E}
_{r^{\prime}}$ and each vector $\mathsf{s}\in N_{x}$ there exists an index
$j\left(  \mathsf{s}\right)  $ for which the value $l_{j\left(  \mathsf{s}
\right)  }\left(  \mathsf{s}\right)  $ of the functional $l_{j\left(
\mathsf{s}\right)  }$ is not only negative but moreover satisfies
\begin{equation}
l_{j\left(  \mathsf{s}\right)  }\left(  \mathsf{s}\right)  <-2c\left\Vert
\mathsf{s}\right\Vert .\label{86}
\end{equation}
Hence for $y=\left(  x,\mathsf{s}_{x}\right)  \in U_{r}\left(  \mathcal{E}
_{r^{\prime}}\right)  $ we have
\begin{equation}
F_{j\left(  \mathsf{s}_{x}\right)  }\left(  y\right)  <F_{j\left(
\mathsf{s}_{x}\right)  }\left(  x\right)  -c\left\Vert \mathsf{s}
_{x}\right\Vert \label{186}
\end{equation}
provided both $r$ and $r^{\prime}$ are small. Therefore
\[
\min_{j}\left\{  F_{j}\left(  y\right)  \right\}  \leq F_{j\left(
\mathsf{s}_{x}\right)  }\left(  y\right)  <F_{j\left(  \mathsf{s}_{x}\right)
}\left(  x\right)  -c\left\Vert \mathsf{s}_{x}\right\Vert =\min_{j}\left\{
F_{j}\left(  x\right)  \right\}  -c\left\Vert \mathsf{s}_{x}\right\Vert \ ,
\]
where the last equality holds since $F_{1}\left(  x\right)  =\ldots
=F_{m}\left(  x\right)  $, so we are done.
\end{proof}

\vskip .2cm
Theorem \ref{lq2} is a straightforward consequence of the next Proposition.

\begin{proposition}
\label{lq1} The point $x=0$ is a sharp local maximum of the function
$\mathsf{F}$ \emph{if} the form
\begin{equation}
\sum_{u=1}^{m}\,\lambda^{u}q_{u}\label{10}
\end{equation}
is negative definite on $E$.

\vskip .2cm
In the special case when all the functions $F_u(x)$, $u=1,\dots,m$, are linear-quadratic,
$F_u$ are sums of linear and quadratic forms, 
\begin{equation}F_u(x)=l_{uj}x^j+q_{ujk}x^j x^k\ ,\label{linquadfun}\end{equation}
the \emph{if} statement becomes the \emph{iff} statement.
\end{proposition}

\begin{proof}
In view of Lemma \ref{prele} we can restrict our search of the maximum of the
function $\mathsf{F}$ to the submanifold $\mathcal{E}$.

\vskip .2cm Note that the plane $E$ is the tangent plane to $\mathcal{E}$ at
the point $0\in\mathcal{E}$, so the coordinate projection of $\mathcal{E}$ to
$E$ introduces the local coordinates on $\mathcal{E}$ in a vicinity of $0.$ As
a result, $\mathcal{E}$ can be viewed as a graph of a function $Z$ on $E$,
$Z\left(  \mathbf{x}\right)  \in\mathbb{R}^{m-1}$ $:$
\[
\mathcal{E}=\left\{  \mathbf{x,z}:\mathbf{x}\in E,\mathbf{z}=\left(
z_{1}\left(  \mathbf{x}\right)  ,\ldots,z_{m-1}\left(  \mathbf{x}\right)
\right)  \right\}  .
\]
This is an instance of the implicit function theorem. The point $\mathbf{x}=0$
is a critical point of all the functions $z_{l}\left(  \mathbf{x}\right)  .$

\vskip .2cm Denote by $M$ the restriction of any of the functions $F_{i}$ to
$\mathcal{E}$. Clearly, it is a smooth function, and the differential $dM$
vanishes at $0\in\mathcal{E}$. So our proposition would follow once we check
that\textbf{ }the second quadratic form of $M$ at $0$ is twice the
form\textbf{ }$\left(  \ref{10}\right)  .$ To see that, let us compute the
derivative $\frac{d^{2}M}{dx_{1}^{2}}$ at the origin; the computation of other
second derivatives repeats this computation. We have
\begin{align*}
&  \frac{d}{dx_{1}}M\left(  \mathbf{x,z}\left(  \mathbf{x}\right)  \right)
=\left( \! \frac{\partial}{\partial x_{1}}M\!\right) \! \left(  \mathbf{x,z}
\left(  \mathbf{x}\right)  \right) \\
& \! +\left( \! \frac{\partial}{\partial x_{n-m+2}}M\!\right) \! \left(
\mathbf{x,z}\left(  \mathbf{x}\right)  \right)  \frac{\partial}{\partial
x_{1}}z_{1}\left(  \mathbf{x}\right)  +\ldots+\left( \! \frac{\partial
}{\partial x_{n}}M\!\right) \! \left(  \mathbf{x,z}\left(  \mathbf{x}\right)
\right)  \frac{\partial}{\partial x_{1}}z_{m-1}\left(  \mathbf{x}\right)  ,
\end{align*}
and then
\begin{align*}
\frac{d^{2}}{dx_{1}^{2}}M\left(  \mathbf{x,z}\left(  \mathbf{x}\right)
\right)  {|}_{\mathbf{x}=0}  &  =2\left[  q_{1}\right]  _{1,1}+\left[
l_{1}\right]  _{1}\cdot0\text{ (since all }\frac{\partial}{\partial x_{1}
}z_{l}\left(  \mathbf{0}\right)  =0\text{)}\\
&  +\left[  l_{1}\right]  _{n-m+2}\cdot\frac{\partial^{2}}{\partial x_{1}^{2}
}z_{1}\left(  \mathbf{0}\right)  +\ldots+\left[  l_{1}\right]  _{n}\cdot
\frac{\partial^{2}}{\partial x_{1}^{2}}z_{m-1}\left(  \mathbf{0}\right)  .
\end{align*}
Let us introduce the vector
\[
\Delta=\left(  0,\ldots,\frac{\partial^{2}}{\partial x_{1}^{2}}z_{1}\left(
\mathbf{0}\right)  ,\ldots,\frac{\partial^{2}}{\partial x_{1}^{2}}
z_{m-1}\left(  \mathbf{0}\right)  \right)  .
\]
Then we have
\[
\frac{d^{2}}{dx_{1}^{2}}M_{1}\left(  \mathbf{x,z}\left(  \mathbf{x}\right)
\right)  {|}_{\mathbf{x}=0}=2\left[  q_{1}\right]  _{1,1}+l_{1}\left(
\Delta\right)  .
\]
Since we have $m-1$ identities
\[
M_{1}\left(  \mathbf{x,z}\left(  \mathbf{x}\right)  \right)  =M_{2}\left(
\mathbf{x,z}\left(  \mathbf{x}\right)  \right)  =M_{m}\left(  \mathbf{x,z}
\left(  \mathbf{x}\right)  \right)  ,
\]
we can write also
\[
\frac{d^{2}}{dx_{1}^{2}}M\left(  \mathbf{x,z}\left(  \mathbf{x}\right)
\right)  {|}_{\mathbf{x}=0}=2\left[  q_{l}\right]  _{1,1}+l_{l}\left(
\Delta\right)  ,\ l=2,\ldots,m.
\]
By $\left(  \ref{32}\right)  $ we then have
\[
\frac{d^{2}}{dx_{1}^{2}}M\left(  \mathbf{x,z}\left(  \mathbf{x}\right)
\right)  {|}_{\mathbf{x}=0}=2\left(  \sum_{l}\lambda^{l}\left[  q_{l}\right]
_{1,1}\right)  ,
\]
so our claim follows.
\end{proof}

\subsubsection{Concluding remark}
We stress that the space
$\mathcal{E}$, on which all the functions $F_{u}\left(  x\right)$, 
$u=1,\ldots,m$, are equal, is a very natural object in the
study of a local maximum of the function $\mathsf{F}(x)$, see (\ref{deffuF}).
As a supporting evidence we provide a simple proof of a weakened form of
Lemma \ref{prele}. 

\begin{lemma}\label{silem} 
Assume that any $m-1$ differentials among $dF_j\left(  0\right)$, $u=1,\ldots,m$, are linearly independent. 
If the equalities $F_{1}\left(  y\right)  =...=F_{m}\left(
y\right)  $ are not satisfied at a point $y\in\mathbb{R}^{n}$ with small
enough $\left\Vert y\right\Vert $ then $y$ cannot be a local maximum of the
function $\mathsf{F}$.
\end{lemma}

\begin{proof}
Suppose that for some $r,$ $1\leq r<m$, we have
\[
F_{1}\left(  y\right)  =...=F_{r}\left(  y\right)  <F_{r+1}\left(  y\right)
\leq...\leq F_{m}\left(  y\right) \ .
\]
Since the linear functionals $dF_{1},\ldots,dF_{r}$ are independent, there
exists a vector $v$ such that all the values $dF_{i}\left(  v\right) $,
$i=1,\ldots,r$, are positive. Therefore all the functions $N_{i}\left(
t\right)  := F_{i}\left(  y+tv\right) $, $i=1,\ldots,r$, are increasing in $t$
at $t=0,$ provided both $\left\Vert y\right\Vert $ and $t$ are small enough.
\end{proof}

\vskip .2cm In \cite{OS}, we were guided by this kind of logic in our search
of the curve $\gamma(\varphi)$: for each $\varphi$ the point $\gamma
(\varphi)\in M^{6}$ is defined by the condition that the distances under
consideration coincide.

\vskip.2cm However, Lemma \ref{silem} does not suffice for the proof of 
Proposition \ref{lq1}, as the following example shows. Let $n=m=2$ and
$$F_1(x_1,x_2)=x_2-x_1^2\ ,\ F_2(x_1,x_2)=2x_2-x_1^2\ .$$
The set $\mathcal{E}$ is the $x_1$-axis. The differential $dF_1(0)$ and $dF_2(0)$ are linearly dependent but the dependency is not convex. The restriction of any of 
functions $F_u$ on $\mathcal{E}$ is the function $-x_1^2$ having the maximum at the origin. However this is not a local maximum
$\mathsf{F}(x_1,x_2)$: for instance, the restriction of the function $\mathsf{F}$ on the $x_2$-axis is the monotone function  
$$\left\{\begin{array}{rl}x_2& \text{for}\ x_2>0\ ,\\[.6em] 2x_2& \text{for}\ x_2\leq0\ .\end{array}\right.$$
 
For the proof of Theorem \ref{lq2}, we need, in the relation $\left(\ref{186}\right)$, a
quantitative statement established in the claim (\ref{85}) of Lemma
\ref{prele}.

\section{Hidden symmetry\vspace{.2cm}}
\label{hiddensymmetry}
In this Section we discuss and generalize the observations from Section \ref{seproo}
about the algebraic nature of the Taylor coefficients for the
deformations around the configuration $C_{6}\left(  \varphi_{\mathfrak{m}},\delta_{\mathfrak{m}},\varkappa_{\mathfrak{m}}\right)$.

\subsection{Galois symmetry\vspace{.1cm}}\label{hiddenGsymmetry}
As we mentioned in the proof from Section \ref{seproo}, the coefficients of the differentials of squares of distances around
the record point $C_{6}\left(  \varphi_{\mathfrak{m}},\delta_{\mathfrak{m}},\varkappa_{\mathfrak{m}}\right)$
belong - after the strange normalization (\ref{strnorma}) -
to the field $\mathbb{Q}[\tau]$. The same holds for the coefficients of all 12 quadratic forms in 15 variables. This looks miraculous. We believe that all coefficients of the Taylor decompositions of all 15 distances (not only 12 relevant distances) around
the record point $C_{6}\left(  \varphi_{\mathfrak{m}},\delta_{\mathfrak{m}},\varkappa_{\mathfrak{m}}\right)$
belong to the field $\mathbb{Q}[\tau]$. We have checked that it is so to some orders of $t$ (for some distances up to $t^8$).

\vskip .2cm
The evidence that the Galois symmetry is global -- that is, that the Galois symmetry holds on the level of functions, not only their Taylor decompositions -- is given in the Subsection \ref{glovevi}, see Proposition \ref{geogasy} and especially formula  (\ref{formdistad}).

\vskip .2cm
We reveal now a hidden symmetry, based on the above observation, of the formulas for the
coefficients of the Taylor expansions of distances around the point $C_{6}\left(  \varphi_{\mathfrak{m}}
,\delta_{\mathfrak{m}},\varkappa_{\mathfrak{m}}\right)$.

\vskip .2cm
Let $\mathcal{W}$ be a $\mathbb{Q}$-vector space with the basis
$$\{ J_\upsilon\}\ \ \text{where}\ \ J\in \{A,B,C,D,E,F\}\ \ \text{and}\ \ \upsilon\in \{\varkappa,\varphi,\delta\} \ .$$
Let $\widetilde{\mathcal{W}}:=\mathbb{Q}[\tau]\otimes_{\mathbb{Q}} \mathcal{W}$. We consider $\widetilde{\mathcal{W}}$
as a vector space over $\mathbb{Q}$. The differentials of distances (see formulas in Subsection \ref{vadifi})
are naturally interpreted as elements of $\widetilde{\mathcal{W}}$. The Galois conjugation $\tau\mapsto \bar\tau$ turns into an involutive automorphism of the space $\widetilde{\mathcal{W}}$ which we denote by $\iota$.

\vskip .2cm
Let $\Pi_\varpi$ and $\Pi_\varrho$ be the operators in $\widetilde{\mathcal{W}}$ realizing the following permutations:
$$\varpi:=(A,B,C)(D,E,F)$$
and
$$\varrho:=(A,D)(B,F)(C,E)\ .$$
It is straightforward to see that the operators $\Pi_\varpi$ and $\Pi_\varrho$
preserve the differentials of distances from Subsection \ref{vadifi}.
The operators $\Pi_\varpi$ and $\Pi_\varrho$ (as well as the permutations $\varpi$ and $\varrho$) generate the group
$\mathbb{D}_3$.

\vskip .2cm
Let $\Pi_\varsigma^\circ$ be the operator in $\widetilde{\mathcal{W}}$ defined by
\begin{equation}\label{neinvo}J_\varkappa\mapsto -\varsigma(J)_\varkappa\ ,\  J_\varphi\mapsto \varsigma(J)_\varphi\ ,\  J_\delta\mapsto -\varsigma(J)_\delta\ ,\end{equation}
for the following permutation $\varsigma$:
$$\varsigma :=(B,C)(D,F)\ ,$$
the elements $A$ and $E$ are fixed by $\varsigma$.
Finally, let $\Pi_\varsigma$ be the composition of the Galois involution and
$\Pi_\varsigma^\circ$,
\begin{equation}\label{neinvo2}\Pi_\varsigma :=\iota\circ\Pi_\varsigma^\circ\ .\end{equation}
The direct inspection shows that the operator $\Pi_\varsigma$ preserves the differentials of distances from Subsection \ref{vadifi};
it sends the differentials of the $\{AF,CE,BD\}$-triplet to the differentials of the $\{CF,BE,AD\}$-triplet and permutes the differentials of the 6-plet $\{AB,BC,CA,DE,EF,FD\}$.

\vskip .2cm
Let $G$ be the group generated by the operators $\Pi_\varpi$, $\Pi_\varrho$ and $\Pi_\varsigma$. We have
$$\Pi_\varpi=\Pi_\varrho \Pi_\varsigma \Pi_\varrho \Pi_\varsigma\ ,$$
so the group $G$ is generated by the operators $\Pi_\varrho$ and $\Pi_\varsigma$ alone. The underlying permutation of
$\Pi_\varrho \Pi_\varsigma$ is
$$(A,F,C,E,B,D)\ ,$$
so
$$ (\Pi_\varrho \Pi_\varsigma)^6=\text{id}\ .$$
Thus, the group $G$ is the dihedral group $\mathbb{D}_6$. Note that the underlying permutations of tangent lines
form exactly the the symmetry group $\mathbb{D}_6$ of the initial configuration $C_6$.

\vskip .2cm
We have checked that the operator $\Pi_\varsigma$ preserves the second differentials as well and we believe that it is so for all orders of the Taylor expansion.
In the same way as for the first differentials, the operator $\Pi_\varsigma$ sends the second differentials of the $\{AF,CE,BD\}$-triplet to the second differentials of the $\{CF,BE,AD\}$-triplet and permutes the second differentials of the 6-plet
$\{AB,BC,CA,DE,EF,FD\}$.
To convince the reader we present three formulas for the second differentials, one differential for each of the triplets
$\{AF,CE,BD\}$ and $\{CF,BE,AD\}$ and one differential for the 6-plet $\{AB,BC,CA,DE,EF,FD\}$:
$$\left[d(BD)^2\right]_2=\frac{2597-1017\sqrt{5}}{127776}(B_{\varkappa,1}+D_{\varkappa,1})^2
-\frac{651+236\sqrt{5}}{792}(B_{\delta,1}^2+D_{\delta,1}^2)$$
$$+\frac{-97+60\sqrt{5}}{528}(B_{\varphi,1}^2+D_{\varphi,1}^2)+\frac{265-3\sqrt{5}}{132}B_{\varphi,1}D_{\varphi,1}
+\frac{219+124\sqrt{5}}{198}B_{\delta,1}D_{\delta,1}$$
$$+\frac{29-109\sqrt{5}}{2904}(B_{\varkappa,1}+D_{\varkappa,1})(B_{\delta,1}+D_{\delta,1})
+\frac{5-48\sqrt{5}}{132}(B_{\varphi,1}B_{\delta,1}+D_{\varphi,1}D_{\delta,1})$$
$$+\frac{181+29\sqrt{5}}{132}(B_{\delta,1}D_{\varphi,1}+B_{\varphi,1}D_{\delta,1})
+\frac{90-17\sqrt{5}}{726}(B_{\varkappa,1}+D_{\varkappa,1})(B_{\varphi,1}+D_{\varphi,1})$$
for the $\{AF,CE,BD\}$-triplet,
$$\left[d(CF)^2\right]_2=\frac{2597+1017\sqrt{5}}{127776}(C_{\varkappa,1}+F_{\varkappa,1})^2
-\frac{651-236\sqrt{5}}{792}(C_{\delta,1}^2+F_{\delta,1}^2)$$
$$-\frac{97+60\sqrt{5}}{528}(C_{\varphi,1}^2+F_{\varphi,1}^2)+\frac{265+3\sqrt{5}}{132}C_{\varphi,1}F_{\varphi,1}
+\frac{219-124\sqrt{5}}{198}C_{\delta,1}F_{\delta,1}$$
$$+\frac{29+109\sqrt{5}}{2904}(C_{\varkappa,1}+F_{\varkappa,1})(C_{\delta,1}+F_{\delta,1})
+\frac{5+48\sqrt{5}}{132}(C_{\varphi,1}B_{\delta,1}+F_{\varphi,1}F_{\delta,1})$$
$$+\frac{181-29\sqrt{5}}{132}(C_{\delta,1}F_{\varphi,1}+C_{\varphi,1}F_{\delta,1})
+\frac{90+17\sqrt{5}}{726}(C_{\varkappa,1}+F_{\varkappa,1})(C_{\varphi,1}+F_{\varphi,1})$$
for the $\{CF,BE,AD\}$-triplet and
$$\frac{150}{11}\left[d(BC)^2\right]_2=\frac{1}{32}(B_{\varkappa,1}-C_{\varkappa,1})^2
+\frac{133+9\sqrt{5}}{16}B_{\varphi,1}^2+\frac{133-9\sqrt{5}}{16}C_{\varphi,1}^2$$
$$+\frac{27-2\sqrt{5}}{16}B_{\varphi,1}(B_{\varkappa,1}-C_{\varkappa,1})
+\frac{27+2\sqrt{5}}{16}C_{\varphi,1}(B_{\varkappa,1}-C_{\varkappa,1})+\frac{109}{4}B_{\varphi,1}C_{\varphi,1}$$
$$+\frac{19+5\sqrt{5}}{4}B_{\delta,1}C_{\varphi,1}-\frac{19-5\sqrt{5}}{4}B_{\varphi,1}C_{\delta,1}
-\frac{53}{3}B_{\delta,1}C_{\delta,1}$$
$$-\frac{103+39\sqrt{5}}{24}B_{\delta,1}^2-\frac{103-39\sqrt{5}}{24}C_{\delta,1}^2
-\frac{43+4\sqrt{5}}{2}B_{\varphi,1}B_{\delta,1}+\frac{43-4\sqrt{5}}{2}C_{\varphi,1}C_{\delta,1}
$$
$$+\frac{7\bar\tau}{4}B_{\delta,1}(B_{\varkappa,1}-C_{\varkappa,1})
-\frac{7\tau}{4}C_{\delta,1}(B_{\varkappa,1}-C_{\varkappa,1})$$
for the 6-plet.

\vskip .3cm
We believe that this action of the group $\mathbb{D}_6$ extends to all orders of the Taylor decompositions of the distances.

\paragraph{Remark.} We were not discussing the remaining $\{AE,BF,CD\}$-triplet because it was not relevant for the proof.
Still, the same phenomenon holds for the Taylor coefficients of the squares of the distances in this triplet. As an illustration we
present the first and second differentials for the squares of the distances between the lines $B$ and $F$:
$$\frac{169}{6}\left[d(BF)^2\right]_1=2\sqrt{5}(B_{\varkappa,1}+F_{\varkappa,1})
+5(B_{\varphi,1}+F_{\varphi,1})-\sqrt{5}(B_{\delta,1}+F_{\delta,1})\ ,$$
and
$$-\frac{105456}{11}\left[d(BF)^2\right]_2=209 (B_{\varkappa,1}+F_{\varkappa,1})^2+560\sqrt{5}
(B_{\varkappa,1}+F_{\varkappa,1})(B_{\varphi,1}+F_{\varphi,1})$$
$$+57445(B_{\varphi,1}+F_{\varphi,1})^2-115600 B_{\varphi,1}F_{\varphi,1}-
404 (B_{\varkappa,1}+F_{\varkappa,1})(B_{\delta,1}+F_{\delta,1})$$
$$+5492\sqrt{5}(B_{\varphi,1}B_{\delta,1}+F_{\varphi,1} F_{\delta,1})-6208\sqrt{5}
(B_{\varphi,1}F_{\delta,1}+B_{\delta,1}F_{\varphi,1} )$$
$$+1466(B_{\delta,1}^2+F_{\delta,1}^2)-968B_{\delta,1}F_{\delta,1}\ .$$

\subsection{Discussion\vspace{.1cm}}\label{glovevi}
Puzzled by the hidden Galois symmetry, described in Subsection \ref{hiddenGsymmetry}, we performed several experiments wondering
whether this hidden symmetry is specific for the record point $C_{6}\left(  \varphi_{\mathfrak{m}},\delta_{\mathfrak{m}},\varkappa_{\mathfrak{m}}\right)$
or it is inherent at some other points of the curve $\gamma(\varphi)$, see formula (\ref{curgamma}), Section \ref{mairesu}.
We computationally clarify this question in the present section.
Namely we explain the geometric origin of the symmetry and show that it becomes the Galois symmetry for `rational' points
of the curve $\gamma(\varphi)$.

\vskip .2cm
We recall some information about the curve $\gamma(\varphi)$ \cite{OS}.

\vskip .2cm
The curve $\gamma(\varphi)$ is related to a part $\Gamma$ of the plane algebraic curve $\Psi =0$ where
\[
\Psi\!=\!4S^{2}-8T^{2}-3S^{4}+29S^{2}T^{2}-4T^{4}-22S^{4}T^{2}+14S^{2}T^{4}
+4S^{6}T^{2}-7S^{4}T^{4}+S^{2}T^{6}\ ,
\]
see Fig. \ref{CurvePsi}.

\begin{figure}[th]
\vspace{.2cm} \centering
\includegraphics[scale=0.5]{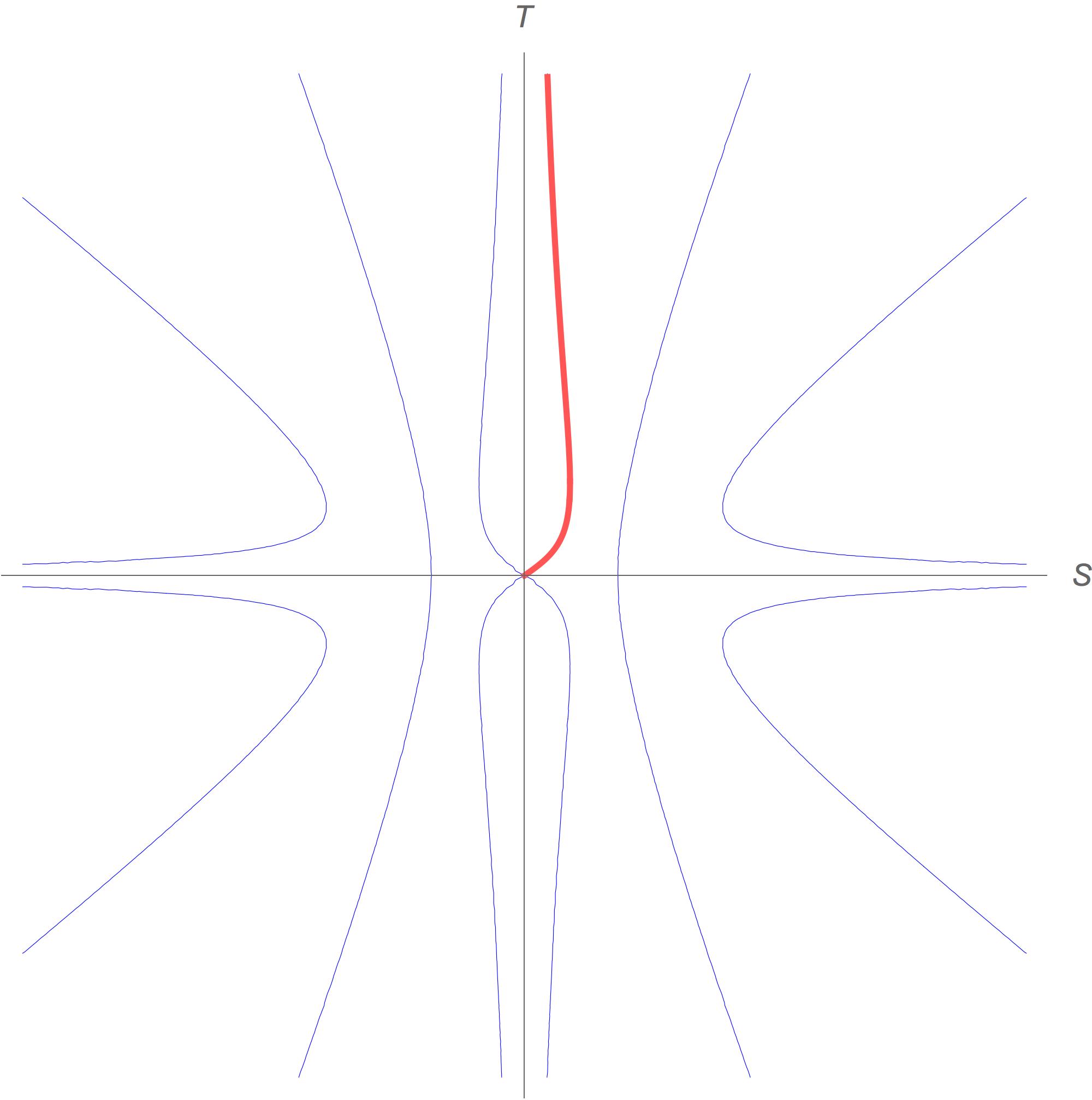}
\vskip -.2cm
\caption{Curve $\Psi=0$, the part $\Gamma$ is depicted in red}
\label{CurvePsi}
\end{figure}

The curve $\Gamma$ admits the following parameterization:
\begin{equation}
S(x)=2\sqrt{\frac{(1-x)x(1+x)}{1+7x+4x^{2}}}\ , \label{trajphi0}
\end{equation}
\begin{equation}
T(x)=\sqrt{\frac{(1-x)(1+3x)}{x(1+7x+4x^{2})}}\ , \label{trajdelta0}
\end{equation}
where $x$ ranges from 1 to 0.

\vskip .2cm
The parameterization of the curve $\gamma(\varphi)$ is given by
\begin{equation}
S(x)\equiv\sin\bigl(\varphi(x)\bigr)
\ ,\ T(x)\equiv\tan\bigl(\delta(x)\bigr)\ ,
\end{equation}
and
\begin{equation}
\tan\bigl(\varkappa(x)\bigr)=\frac{x-1}{\sqrt{(1+x)(1+3x)}}\ . \label{trajkappa}
\end{equation}
For brevity, we denote by $C_{6,x}$ the configuration $C_{6}\bigl(  \varphi(x),\delta(x),\varkappa(x)\bigr)$ of six tangent lines.
The squares of the relevant twelve distances between the lines of the configuration $C_{6,x}$ are all equal to
$$\frac{12x}{1+7x+4x^2}\ .$$

In our experiments we were fixing various rational values of the parameter $x$, then making a general (involving all 15 parameters) perturbation of the
configuration $C_{6,x}$ and studying the nature and the structure of the Taylor coefficients of the squares of distances in a vicinity of $C_{6,x}$.

\vskip .2cm
For example, at $x=1/3$, the Taylor coefficients, after an appropriate normalization, belong to $\mathbb{Q}[\sqrt{2}]$ and the Galois automorphism
of  $\mathbb{Q}[\sqrt{2}]$ restores the $\mathbb{D}_6$ symmetry, as in Section \ref{hiddenGsymmetry}. However, at $x=1/5$, the Taylor coefficients,
after an appropriate normalization, belong to $\mathbb{Q}$.

\vskip .2cm
To summarize the results of our study, let
\begin{equation}\label{esseirra}
\mathfrak{p}_x =\sqrt{\frac{(1 + x) (1 + 3 x)}{3}}\ .
\end{equation}

As in Section \ref{seproo} the perturbed position of a line $J\in \{A,B,C,D,E,F\}$
in the configuration $C_{6,x}$ is
$$J=J(\varkappa(x)+\Delta_J^{\varkappa},\varphi(x)+\Delta_J^{\varphi},\delta(x)+\Delta_J^{\delta})$$
and
\begin{equation}\label{strnormawithx}
\begin{array}{c}
\displaystyle{ \Delta_J^{\varkappa}=\vartheta_\varkappa\cdot \left(J_{\varkappa,1}t+J_{\varkappa,2}t^2+o(t^2)\right)\ ,}\\[1.5em]
\displaystyle{ \Delta_J^{\varphi}=\vartheta_\varphi\cdot \left(J_{\varphi,1}t+J_{\varphi,2}t^2+o(t^2)\right)\ ,}\\[1.5em]
\displaystyle{ \Delta_J^{\delta}=\vartheta_\delta\cdot \left(J_{\delta,1}t+J_{\delta,2}t^2+o(t^2)\right)\ .}
\end{array}\end{equation}
Here $\vartheta_\varkappa, \vartheta_\varphi$ and $\vartheta_\delta$ are normalization constants.

\vskip .2cm
As before, to fix the rotational symmetry we keep the tangent line $A$ at its place, that is,
$A_{\varphi,j}=A_{\varkappa,j}=A_{\delta,j}=0$, $j=1,2,\dots$.

\begin{proposition}\label{geogasy} Let $x$ be a rational number between 0 and 1 such that $\mathfrak{p}_x $ is not rational.

\vskip .2cm
\noindent {\rm (i)} There exists a choice of the normalization constants $\vartheta_\varkappa, \vartheta_\varphi$ and $\vartheta_\delta$ such that
the Taylor coefficients  of the squares of distances belong to $\mathbb{Q}[\mathfrak{p}_x ]$.

\vskip .2cm
\noindent {\rm (ii)} The operator $\Pi_\varsigma $, defined by the formula (\ref{neinvo2}), where $\iota$ is the Galois conjugation $\mathfrak{p}_x \to -\mathfrak{p}_x $
of $\mathbb{Q}[\mathfrak{p}_x ]$, restores the $\mathbb{D}_6$ symmetry.
\end{proposition}
\noindent{\bf Note.} For $x=1/5$ we find that $\mathfrak{p}_x $ is rational, $\mathfrak{p}_x =4/5$. The equation $\mathfrak{p}_x=\mathfrak{r}$, $x\in\mathbb{Q}$ and
$\mathfrak{r}\in\mathbb{Q}$, can be easily solved: substituting $x=(y-2)/3$ we find $y^2=9\mathfrak{r}^2+1$, so the question reduces to Pythagorean triples.

\vskip .2cm
The irrationality $\mathfrak{p}_x$ can be found among the geometric objects in the configuration $C_{6,x}$. This irrationality is related to the cosine $\gamma_x$ of the angle between the tangent lines $A$ and $D$. We calculate:
$$\gamma_x =-\frac{3(1-x)}{2(1+2x)}\,\mathfrak{p}_x
+\frac{x(1+5x)}{2(1+2x)}\ .$$
In particular, under the conditions of Proposition \ref{geogasy}, $\mathbb{Q}[\mathfrak{p}_x ]=\mathbb{Q}[\gamma_x]$ and we can reformulate the part (i) of
Proposition \ref{geogasy}: the coefficients of the Taylor expansions belong to the field $\mathbb{Q}[\gamma_x]$.

\vskip .2cm
The cosine of the angle between the tangent lines $A$ and $F$ is equal to $\overline{\gamma}_x$, the Galois conjugate of $\gamma_x$ in the
field $\mathbb{Q}[\mathfrak{p}_x]$,
$$\overline{\gamma}_x =\frac{3(1-x)}{2(1+2x)}\,\mathfrak{p}_x
+\frac{x(1+5x)}{2(1+2x)}\ .$$

\vskip .2cm
\noindent Sketch of the {\bf Proof} of Proposition \ref{geogasy}. We do not furnish the complete details in order not to overload the presentation.
Namely, we will work with only two distances, $d(AD)$ and $d(AF)$.  Since we keep the line $A$ fixed (which makes the formulas more tangible), the
distance $d(AD)$ depends, for a given $x$, on the three parameters, characterizing the position of the perturbed line $D$. We consider only a simplified
situation, namely we perturb only the angle\footnote{$\ $Due to the character of the formulas, the angle $\delta$ is the most manageable of
the three angles. We will write the formula for the distance between the line $A$ and the perturbed line $D$ `im gro\ss en', without decomposing in the Taylor series.}
$\delta$  of the line $D$ and follow the dependence of the distance on this variation. Thus, for a given $x$, we
consider the function $d\bigl(A,D(\xi)\bigr)$ where
$$A=A(\varkappa(x),\varphi(x),\delta(x))\ \ \text{and}\ \
D(\xi)=D(\varkappa(x),\varphi(x),\delta(x)+\xi)\ .$$

\vskip .2cm
Since the tangent function has rational Taylor coefficients, we are allowed to pass to the variable
$$\Xi=\tan (\xi)$$
instead of $\xi$.

\vskip .2cm
In the notation of Section \ref{sectConfManifold}, let  $u=\bigl(  (\varphi_1,\varkappa_1),\uparrow_{\delta_1}\bigr)  $ and
$v=\bigl(  (\varphi_2,\varkappa_2),\uparrow_{\delta_2}\bigr)$ be two lines in $M$.
The formula (\ref{formdist}) for the square of the distance between lines $u$ and $v$ has the following explicit form in the coordinates
$(\varphi,\varkappa,\delta)$
\begin{equation}\label{explformdist}
d_{uv}^{\hspace{.04cm}2}=\frac{\mathcal{N}^{\hspace{.04cm}2}}{\bigl(1+\tan^2 (\delta_1)\bigr)\bigl(1+\tan^2 (\delta_2)\bigr)
-\mathcal{D}^{\hspace{.04cm}2}}\ ,\end{equation}
where
$$\begin{array}{c}
\mathcal{N}\!:=\!\bigl(\tan(\delta_1)+\tan(\delta_2)\bigr)
\Bigl(
\cos(\varphi_1)\cos(\varphi_2)-\cos(\Delta\varkappa)
\bigl(
1-\sin(\varphi_1)\sin(\varphi_2)\bigr)
\Bigr)\\[1em]
-\bigl(1-\tan(\delta_1)\tan(\delta_2)\bigr)\sin(\Delta\varkappa)\bigl(\sin(\varphi_1)-\sin(\varphi_2)\bigr)
\ ,\end{array}$$

$$\begin{array}{c}\mathcal{D}:=
\cos(\varphi_1)\cos(\varphi_2)
+\cos(\Delta\varkappa)\bigl(
\sin(\varphi_1)\sin(\varphi_2)+\tan(\delta_1)\tan(\delta_2)\bigr)\\[1em]
+\sin(\Delta\varkappa)\bigl(
\tan(\delta_2)\sin(\varphi_1)-\tan(\delta_1)\sin(\varphi_2)\bigr)\
\end{array}$$
and
$$\Delta\varkappa=\varkappa_1-\varkappa_2\ .$$

\paragraph{Assertion {\rm (i)}.} Let
\begin{equation}\label{renorx}
\Xi_1=\frac{\mathfrak{q}_x}{\mathfrak{p}_x^2}\,\Xi\ ,
\end{equation}
where
\begin{equation}\label{renorx2}
\mathfrak{q}_x=\sqrt{\frac{1+x}{3x(1-x)(1+7x+4x^2)}}\ .
\end{equation}

\vskip .2cm
The formula (\ref{explformdist}) for the lines $A$ and  $D(\xi)$ gives, after numerous simplifications:
\begin{equation}\label{formdistad}
d\bigl(A,D(\xi)\bigr)^2=\frac{x(1-x)(1+3x)}{1 + 7 x + 4 x^2}\ \cdot  \frac{\mathfrak{n}_{AD}^2}{\mathfrak{s}-\mathfrak{t}_{AD}^2}\ ,\end{equation}
where
$$\mathfrak{s}=1+\Xi^2=x(1-x)(1+3x)(1+7x+4x^2)\,\Xi_1^2\ ,$$
and
$$\mathfrak{n}_{AD}=\frac{3}{x}\left(\frac{2\mathfrak{p}_x\gamma_x}{1+3x}+\frac{1}{1+2x}\right)+\left( (1+2x)^2 (\gamma_x-1)+6x\gamma_x\right)
\mathfrak{p}_x\Xi_1\ ,$$
$$\mathfrak{t}_{AD}=\gamma_x-3(1-x)\left( \mathfrak{p}_x\gamma_x+\frac{1+3x}{2(1+2x)}\right)\mathfrak{p}_x\Xi_1\ .$$
The formula (\ref{formdistad}) establishes the part (i). Indeed, the formula (\ref{renorx}) gives the needed renormalization such that the final
expression  (\ref{formdistad}) is a rational function in $\Xi_1$ with coefficients in $\mathbb{Q}[\mathfrak{p}_x]$.

\paragraph{Assertion {\rm (ii)}.}
A parallel computation, now for the tangent lines $A$ and  $F(\xi)=F(\varkappa(x),\varphi(x),\delta(x)+\xi)$, yields
$$d\bigl(A,F(\xi)\bigr)^2=\frac{x(1-x)(1+3x)}{1 + 7 x + 4 x^2}\ \cdot  \frac{\mathfrak{n}_{AF}^2}{\mathfrak{s}-3(1-x^2)\,\mathfrak{t}_{AF}^2}\ ,$$
where
$$\mathfrak{n}_{AF}=\frac{3}{x}\left(-\frac{2\mathfrak{p}_x\overline{\gamma}_x}{1+3x}+\frac{1}{1+2x}\right)-\left( (1+2x)^2 (\overline{\gamma}_x-1)
+6x\overline{\gamma}_x\right)\mathfrak{p}_x\Xi_1\ ,$$
$$\mathfrak{t}_{AF}=\overline{\gamma}_x+3(1-x)\left( -\mathfrak{p}_x\overline{\gamma}_x+\frac{1+3x}{2(1+2x)}\right)\mathfrak{p}_x\Xi_1\ .$$
A direct comparison shows that the expressions for $d\bigl(A,F(\xi)\bigr)^2$ are obtained from the expressions for $d\bigl(A,D(\xi)\bigr)^2$ by the simultaneous change of sign of $\Xi$ and
$\mathfrak{p}_x $, in the full accordance with the formula (\ref{neinvo}), so the Galois action of the operator $\Pi_\varsigma$, see (\ref{neinvo2}), restores
the $\mathbb{D}_6$ symmetry.

\vskip .2cm
In the general situation, when all three parameters, $\varkappa,\varphi$ and $\delta$ of the tangent lines $D$ and $F$ are perturbed, the closed formulas
for the squares of distances are considerably lengthier (and as non-illustrative as the formula (\ref{formdistad})) and we do not present them.
\hspace{.15cm}\myblacksquare

\vskip .2cm
Two phenomena exhibited in Proposition \ref{geogasy} -- (i) all irrationalities, except one, are absorbed in the normalization factors;  (ii) the Galois conjugation of
the remaining irrationality restores the $\mathbb{D}_6$ symmetry -- shows a certain consonance between the curve $\Gamma$ and the ingredients of the formula
for the distance between skew lines. In the process of calculations it was important that for a rational $x$ all angles $\varphi(x)$, $\varkappa(x)$ and
$\delta(x)$ are purely geodetic in the sense of \cite{CRS}, that is, squares of trigonometric functions of these angles are rational.

\vskip .2cm
We conclude by two remarks.

\vskip .2cm
\noindent{\bf Remark 1.} The Galois symmetry can be extended to all values of $x$ if one works with the functional fields.
We demonstrate it in the same simplified situation, as in the proof of Proposition \ref{geogasy}, where we perturb only the angle $\delta$  of the lines $D$ and $F$.
Let $\mathbb{F}:=\mathbb{Q}(x)$ be the field of rational functions in one variable. We consider
its biquadratic extension $\mathbb{F}[\mathfrak{q}_x,\mathfrak{p}_x]$, where $\mathfrak{p}_x$ is defined by the formula (\ref{esseirra}) and
$\mathfrak{q}_x$ --  by the formula (\ref{renorx2}). The Galois group of the extension $\mathbb{F}[\mathfrak{q}_x,\mathfrak{p}_x]$ of the field
$\mathbb{F}$ is the Klein four-group $C_2\times C_2$ generated by the sign changes of $\mathfrak{q}_x$ and $\mathfrak{p}_x$. The Galois involution
$\iota\colon\mathfrak{p}_x\to -\mathfrak{p}_x$ is continuous in the natural sense; it allows to define the involution  $\Pi_\varsigma$ which restores the
$\mathbb{D}_6$ symmetry. Thus, performing the involution $\Pi_\varsigma$ on the level
of the functional fields and only then specializing the value of $x$ allows to see a shadow of the continuous extension of the involution $\Pi_\varsigma$ to
all points of the curve $\Gamma$.

\vskip .2cm
\noindent{\bf Remark 2.} We have added this comment because of some questions raised during our talks on the subject.
The parameterization (\ref{trajphi0})-(\ref{trajdelta0}) serve only the part $\Gamma$ of the curve $\Psi=0$. What can be said about other components?

\begin{figure}[th]
\vspace{.2cm} \centering
\includegraphics[scale=0.36]{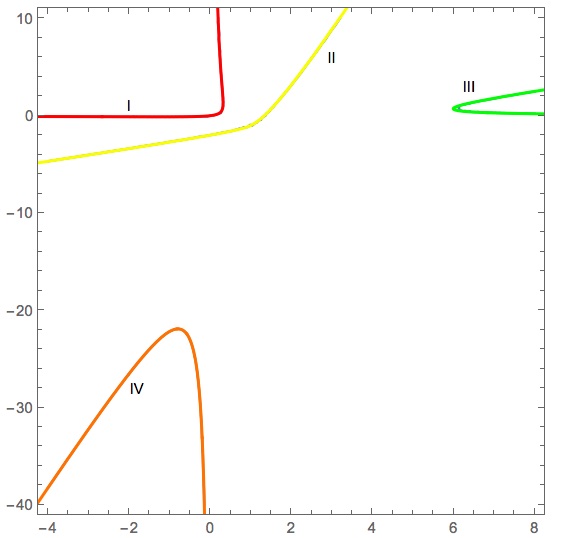}
\vskip -.2cm
\caption{Real components of the curve $\psi=0$}
\label{ParameterizationCurvePsi}
\end{figure}

It is easier to work with the curve $\psi=0$, where
\[
\psi\!=\!4s-8t-3s^{2}+29st-4t^{2}-22s^{2}t+14st^{2}
+4s^{3}t-7s^{2}t^{2}+st^{3}\ ,
\]
since the polynomial $\Psi$ depends only on $s=S^2$ and $t=T^2$. The real components of the curve $\psi=0$ are shown on Fig. \ref{ParameterizationCurvePsi}.

\vskip .2cm
The parameterization (\ref{trajphi0})-(\ref{trajdelta0}) becomes
$$
s(x)=\frac{4(1-x)x(1+x)}{1+7x+4x^{2}}\ \ ,\ \  t(x)=\frac{(1-x)(1+3x)}{x(1+7x+4x^{2})}\ .
$$

The denominators of the rational functions $s(x)$ and $t(x)$ are singular at $x=0$ and the roots $(-7 -\sqrt{33})/8\approx -1.5931$ and
$(-7 +\sqrt{33})/8\approx -0.1569$ of the polynomial $1+7x+4x^{2}$.
For $x$ ranging from $-\infty$ to $(-7 -\sqrt{33})/8$ we obtain the part III (in green) on Fig. \ref{ParameterizationCurvePsi}; for $x$ ranging from $(-7 -\sqrt{33})/8$ to
 $(-7 +\sqrt{33})/8$ -- the part II (in yellow) on Fig. \ref{ParameterizationCurvePsi};
for $x$ ranging from $(-7 +\sqrt{33})/8$ to $0$ -- the part IV (in orange) on Fig. \ref{ParameterizationCurvePsi}; finally, for $x$ ranging from $0$ to
$+\infty$ we obtain the part I (in red) on Fig. \ref{ParameterizationCurvePsi}.

\vskip .2cm
The only singular point of the (homogenized) curve $\psi=0$ in the complex domain is the triple point $(s=1,t=-1)$; it is the image of three points on the
complex $x$-plane, namely, the point $x=-1/2$ and two other points $x=(-1\pm i\sqrt{-7})/4$.

\vskip .5cm\noindent{\footnotesize
{\textbf{Acknowledgements.}
Part of the work of S. S. has been carried out in the framework
of the Labex Archim\`ede (ANR-11-LABX-0033) and of the A*MIDEX project (ANR-11-
IDEX-0001-02), funded by the Investissements d'Avenir French Government program
managed by the French National Research Agency (ANR). Part of the work of S. S. has
been carried out at IITP RAS. The support of Russian Foundation for Sciences (project
No. 14-50-00150) is gratefully acknowledged by S. S. The work of O. O. was supported by
the Program of Competitive Growth of Kazan Federal University and by the grant RFBR
17-01-00585.}}

\end{document}